\documentclass{siamart190516}

\usepackage{algorithm}
\usepackage{algorithmic}
\usepackage{amsmath}
\usepackage{amsfonts}
\usepackage{amssymb}
\usepackage{color}
\usepackage{enumerate}
\usepackage{graphicx}
\usepackage{lscape}
\usepackage{longtable}
\usepackage{multirow}
\usepackage{rotating}
\usepackage{subfigure}
\usepackage{url}

\newtheorem{assumption}{Assumption}[section]

\def\real{\mathbb{R}}

\def\eps{\epsilon}

\DeclareMathOperator{\ared}{ared}
\DeclareMathOperator{\pred}{pred}

\DeclareMathOperator{\rared}{rared}

\newcommand{\revised}[1]{\textcolor{black}{#1}}

\title{A Nonmonotone Matrix-Free Algorithm for Nonlinear Equality-Constrained 
 Least-Squares Problems
\thanks{Version of \today.}}
\author{E. Bergou\thanks{Mohammed VI Polytechnic University, Ben Guerir, Morocco ({\tt elhoucine.bergou@um6p.ma}). 
}
\and Y. Diouane \thanks{ISAE-SUPAERO, Universit\'e de Toulouse, 31055 Toulouse Cedex 4, France
  (\texttt{youssef.diouane@isae-supaero.fr}).}
\and V. Kungurtsev \thanks{Department of Computer Science, Faculty of Electrical Engineering, Czech Technical University in Prague. Support for this author was provided by the OP VVV project
CZ.02.1.01/0.0/0.0/16\_019/0000765 Research Center for Informatics., 
  (\texttt{vyacheslav.kungurtsev@fel.cvut.cz}).}
  \and C. W. Royer \thanks{LAMSADE, CNRS, Universit\'e Paris-Dauphine, Universit\'e PSL, 75016 Paris, France. Support for this author was partially provided by Subcontract 3F-30222 from Argonne National Laboratory, 
  (\texttt{clement.royer@dauphine.psl.eu}).}}

\begin{document}

\maketitle

\begin{abstract}
Least squares form one of the most prominent classes of optimization problems, 
with numerous applications in scientific computing and data fitting. When 
such formulations aim at modeling complex systems, the optimization process 
must account for nonlinear dynamics by incorporating constraints. In addition, 
these systems often incorporate a large number of variables, 
which increases the difficulty of the problem, and motivates the need for 
efficient algorithms amenable to large-scale implementations.

In this paper, we propose and analyze a Levenberg-Marquardt algorithm for 
nonlinear least squares subject to nonlinear equality constraints. Our 
algorithm is based on inexact solves of linear least-squares problems, that
only require Jacobian-vector products. Global convergence is 
guaranteed by the combination of a composite step approach and a 
nonmonotone step acceptance rule. We illustrate the performance of our 
method on several test cases from data assimilation and 
inverse problems: our algorithm is able to reach the vicinity of a 
solution from an arbitrary starting point, and can outperform the most 
natural alternatives for these classes of problems.
\end{abstract}


\section{Introduction}
\label{sec:intro}

\revised{In this paper we are interested in solving least-squares optimization problems
wherein a set of unknown parameters
is sought such that it minimizes the Euclidean norm of a residual vector function.}
This objective is particularly well suited to represent the discrepancy between a 
model and a set of observations. As a result, such formulations have been 
successfully applied to a wide range of applications across 
disciplines~\cite{PCHansen_VPereyra_GScherer_2012}. This is not only due to 
the ubiquitous nature of this problem, but also to the existence of 
efficient optimization algorithms dedicated to solving this problem by 
exploiting its specific structure. In particular, the unconstrained setting 
is particularly well understood: linear least squares can be efficiently
solved by exploiting linear algebra solvers~\cite{SBoyd_LVandenberghe_2018} 
while nonlinear least squares are classically tackled using variants of the 
Gauss-Newton paradigm~\cite{JNocedal_SJWright_2006}. 

In this paper, we consider nonlinear least-squares problems subject to 
nonlinear constraints of the following form:
\begin{equation} \label{eq:prb}
\begin{array}{rl}
\displaystyle  
\min_{x \in \real^d} & f(x) \triangleq \frac{1}{2} \|F(x)\|^2 
= \frac{1}{2}  \sum_{i=1}^m F_i(x)^2, \\
\mbox{s. t.}& C(x) = 0,
\end{array}
\end{equation}
where $\| \cdot \|$ will denote the Euclidean norm,  $F: \real^d \rightarrow \real^m$  and $C: \real^d \rightarrow \real^p$ 
will be assumed to be nonlinear, potentially nonconvex, continuously 
differentiable functions. Although one possible approach to solving this 
problem consists in incorporating the constraints directly into the objective, 
we are mainly interested in situations in which the constraints represent 
physical phenomena that drive the behavior of the underlying system, and should 
not be treated as additional residual functions. Such formulations arise when 
solving inverse problems~\cite{ATarantola_2005} in variational modeling for 
meteorology, such as 4DVAR~\cite{YTremolet_2007}, the dominant data assimilation 
least-squares formulation used in numerical weather prediction centers. 
Similar applications include seismic imaging and fluid 
mechanics~\cite{HAntil_DPKouri_MDLacasse_DRidzal_2016EDS}.

In this work, we aim at combining the Levenberg-Marquardt method~\cite{KLevenberg_1944,
DMarquardt_1963}, one of the most popular algorithms for the solution of nonlinear least squares, 
with nonlinear programming techniques tailored to solving large-scale equality-constrained problems. 
Our framework is inspired by inexact trust-region sequential quadratic programming 
(SQP)~\cite{JrJEDennis_MElAlem_MCMaciel_1997,heinkenschloss2014matrixfreeSQP}, which uses
inexact computation of a composite step in a careful way that does not jeopardize the convergence guarantees. We 
combine these techniques with a nonmonotone rule for accepting or rejecting the step, that eschews 
the use of a penalty function while still guaranteeing global convergence~\cite{MUlbrich_SUlbrich_2003}. A
nonmonotone rule allows greater flexibility in step acceptance, and thus in practice often converges
faster since it accepts a wider range of steps. Our 
approach builds on previously proposed algorithms of trust-region type, that have a broader 
scope than least squares. In our case, we do not use second-order derivatives (note that 
those could be unavailable for algorithmic use), \revised{but rather} take advantage of the second-order 
nature of the Gauss-Newton model. We also allow for a ``matrix-free'' implementation of 
our algorithm, that only requires access to products of the Jacobian matrices (of both the 
constraints and objective) with vectors.
 
The existing literature on Levenberg-Marquardt methods for constrained 
least squares has mainly focused on establishing local convergence~\cite{RBehling_AFischer_2012,izmailov2019globally} 
or complexity guarantees~\cite{NMarumo_TOkuno_ATakeda_2020} for the proposed 
method. In particular, the work of Izmailov et al.~\cite{izmailov2019globally} 
is concerned with local convergence properties of Tikhonov-type regularization 
algorithms, which requires the use of second derivatives. 
Meanwhile, other regularization techniques for constrained least squares that 
do not directly belong to the Levenberg-Marquardt class of methods have recently 
been proposed, that are also based on the SQP methodology 
(see~\cite{ZFLi_MROsborne_TPrvan_2002,DOrban_ASSiquiera_2020} and references 
therein). However, to the best of our knowledge, these approaches are not based 
on nonmonotone rules, and generally require second-order derivatives. 

The rest of this paper is organized as follows. In section~\ref{sec:algo}, we 
detail the various components of our proposed algorithm, with a focus on  
inexactness conditions and nonmonotone rules. 
Section~\ref{sec:globcv} addresses the global convergence of our algorithm. The 
practical behavior of our method is investigated in Section~\ref{sec:numerics} 
on problems including nonlinear data assimilation and inverse PDE-constrained 
optimization. Conclusions and perspectives are finally provided in 
Section~\ref{sec:conc}.  

\paragraph{Notations}
Throughout $\| \cdot \|$ will denote the vector or matrix $l_2$-norm 
in any space of the form $\real^d$, and 
$\langle ., .\rangle$ the associated dot product. We will use $A^\top$ to 
for the transpose of the matrix $A$. The identity matrix in 
$\real^{d \times d}$ will be denoted by $I_d$.

\section{Algorithmic framework}
\label{sec:algo}

In this section, we provide a detailed description of the algorithm studied in 
this paper. Our method combines the Levenberg-Marquardt algorithm with the 
classical Byrd-Omojokun SQP step~\cite[Chapter 18]{JNocedal_SJWright_2006} 
for equality constraints. At every iteration $j$, the method 
computes a trial step of the form $s_j \triangleq n_j+t_j$, with $n_j$ 
being an inexact quasi-normal step (see Section~\ref{subsec:algo:normal}) 
and $t_j$ being an inexact tangential step (see 
Section~\ref{subsec:algo:tangential}). The former aims at improving 
feasibility, while the latter focuses on lowering the objective value while 
retaining feasibility. This step is then accepted or rejected depending on 
nonmonotone decrease requirements: those conditions are described 
in Section~\ref{subsec:algo:nonmonotone}. The full description of the method 
is given in Section~\ref{subsec:algo:fullalgo}.

\subsection{Inexact quasi-normal step}
\label{subsec:algo:normal}

Let $\phi(x)\triangleq \tfrac{1}{2}\|C(x)\|^2$ denote the constraint 
violation function at the point $x$. 
The quasi-normal step is defined as the solution of the following 
minimization subproblem:
\begin{eqnarray}\label{normalstep}
\displaystyle \min_{n \in \real^d} m_j^c(n) 
& \triangleq 
& \frac{1}{2}\| C_j+J^{c}_j n\|^2 + \frac{1}{2} \gamma_j \|n\|^2, 
\end{eqnarray}
where $C_j \triangleq C(x_j)$ and $J^{c}_j \triangleq J^c(x_j)$ denotes the 
Jacobian of the constraints $C$ at $x_j$. The first term in~\eqref{normalstep} 
is the Gauss-Newton model of the constraint violation function, while the 
second term represents a regularization of Levenberg-Marquart type: the 
regularization parameter $\gamma_j$ will be chosen adaptively during the 
algorithmic process.

For practical purposes (and in particular in a large-scale setting), we 
consider an approximate solution of the subproblem~\eqref{normalstep}, 
typically computed using Krylov subspace methods~\cite{EBergou_YDiouane_VKungurtsev_2017,EBergou_SGratton_LNVicente_2016}. 
More precisely, instead of solving
~\eqref{normalstep} to global optimality, we only require that 
our inexact quasi-normal step satisfies:
\begin{equation}\label{eq:decrease:c} 
m_j^c(0) - m_j^c(n_j) \ge \kappa_1\frac{\|C_j\|^2}{\|J^{c}_j\|^2+\gamma_j}.
\end{equation}
for some constant $\kappa_1>0$. Note that the Cauchy point, i.e. the 
minimizer of $m_j^c(n)$ in the subspace spanned by $\nabla m_j^c(0)$, 
satisfies the condition~\eqref{eq:decrease:c} under appropriate 
assumptions~\cite{EBergou_YDiouane_VKungurtsev_2017}. 

\subsection{Inexact tangential step}
\label{subsec:algo:tangential}

Having computed the (inexact) quasi-normal step, we now seek an additional 
step $t_j$ that results in a sufficient decrease of the objective function
in the tangent space along the level sets of the
constraints. We thus consider the Lagrangian function associated with 
problem~\eqref{eq:prb}:
\begin{eqnarray*}
	\mathcal{L}(x,y) &\triangleq &f(x) + y^{\top}C(x),~~~~~~~y \in \real^p.
\end{eqnarray*}
We also define the exact regularized Gauss-Newton model of the Lagrangian 
at the $j$-th iteration by
\begin{equation} \label{eq:modelL}
	m_j^l(s) \; \triangleq \;
	\frac{1}{2}\| F_j+J^F_j s\|^2 +y^{\top}_j J^c_j s 
	+ \frac{1}{2} \gamma_j \|s\|^2+ y^{\top}_j C(x_j),
\end{equation}
where $F_j\triangleq F(x_j)$, $J^F_j\triangleq J^f(x_j)$ denotes the 
Jacobian of the residual function $F$ at the current iterate $x_j$, and 
the vector $y_j\in \real^p$ is a current estimate for the 
Lagrange multiplier. 

The exact tangent step is given as the solution of
\begin{equation}\label{tangentstep}
\begin{array}{rl}
\displaystyle \min_{t \in \real^d} & m_j^l(n_j +t), \\
\text{s.t.} & J^c_j t = 0,
\end{array}
\end{equation}
Note that for any $t \in \real^d$, the objective of~\eqref{tangentstep} decomposes as
\begin{equation*}
m_j^l(n_j+t) \; = \; m_j^l(0) + \frac{1}{2} \langle H_j n_j,  n_j \rangle + \frac{1}{2} \langle H_j t,  t \rangle +   \langle  \nabla \mathcal{L}(x_j,y_j) ,n_j \rangle + \langle  g_j,t \rangle.
\end{equation*}
with $H_j \triangleq {J_j^F}^{\top}J_j^F  + \gamma_j I_m$, 
$ \nabla \mathcal{L}(x_j,y_j)  = {J_j^F}^{\top} F_j +{J^c_j}^{\top}y_j $ and 
$g_j \triangleq \nabla \mathcal{L}(x_j,y_j) +  H_j n_j$. 

Problem~\eqref{tangentstep} can be reformulated as an unconstrained linear least-squares 
problem. Indeed, let $W_j \triangleq W(x_j) \in \real^{d \times d}$ denote a projection 
matrix onto the null space of $J^c(x)$: then $W_j=W_j^\top=W_j^2$ and, for any $t$ such 
that $J^c_j t=0$, there exists $w \in \real^d$ such that $t=W_jw$. Using the reformulation 
$t=W_jw$, the minimization subproblem~\eqref{tangentstep} is thus equivalent to
\begin{equation}\label{tangentstep2}
\begin{array}{rl}
\displaystyle \min_{w \in \real^d} & 
\frac{1}{2} \langle W_j^{\top}H_j W_jw, w  \rangle + \langle W_j^{\top} g_j,w \rangle
\end{array}
\end{equation}
To compute the exact tangential step, one can compute a solution $w^*$ 
of~\eqref{tangentstep2} then apply $W_j$ to obtain the solution $t^*=W_j w^*$ 
of~\eqref{tangentstep}.

One challenge of the approach above is the computation of the matrix $W_j$, which can 
be prohibitive in a large-scale environment. We will thus adopt a matrix-free 
approach~\cite{heinkenschloss2014matrixfreeSQP}: given $w \in \real^d$, 
the vector $t=W_jw$ can be computed by solving the following augmented system
\begin{equation} \label{compute_W_j_exactly}
\begin{pmatrix} I & {J_j^c}^{\top} \\  J_j^c & 0 \end{pmatrix} 
\begin{pmatrix} t\\  z \end{pmatrix} 
= 
\begin{pmatrix} w\\  0 \end{pmatrix}.
\end{equation}
As long as ${J^c_j}$ is surjective, the linear system~\eqref{compute_W_j_exactly} 
possesses a solution. An \emph{inexact solve} of the linear system~\eqref{compute_W_j_exactly} 
corresponds to applying an approximation $\widetilde W_j(\cdot)$ of $W_j$ instead of 
applying the projection matrix $W_j$. Therefore, it can be 
shown~\cite{heinkenschloss2014matrixfreeSQP} that such an inexact solve 
of~\eqref{tangentstep2} corresponds to an exact solve of the following subproblem:
\begin{equation}\label{tangentstep2_approx}
\begin{array}{rl}  
\displaystyle \min_{w \in \real^d} & \frac{1}{2} \langle H_j \widetilde W_j(w),  
\widetilde W_j (w)  \rangle + \langle  \widetilde W_j (g_j), \widetilde W_j (w) \rangle,
\end{array}
\end{equation}
which is itself equivalent to 
\begin{equation}\label{tangentstep_approx}
\begin{array}{rl}
\displaystyle \min_{\widetilde t \in \real^d} & \frac{1}{2} \langle  H_j {\widetilde t}, {\widetilde t}  \rangle + \langle \widetilde W_j (g_j),{\widetilde t} \rangle, 
\end{array}
\end{equation}
\revised{with the change of variables ${\widetilde t}=\widetilde W_j(w)$.}
The approximate solutions of (\ref{tangentstep2_approx}) and (\ref{tangentstep_approx}) are respectively denoted by $w_j$ and \revised{${\widetilde t_j}=\widetilde W_j (w_j)$}. 
The quality of the inexact step ${\widetilde t_j}$ will be measured by the inexact model:
\begin{equation*}
\widetilde m^l_j(n_j+ \widetilde t_j) \triangleq   m_j^l(0) + \frac{1}{2} \langle H_j n_j,  n_j \rangle + \frac{1}{2} \langle H_j \widetilde t_j, \widetilde  t_j \rangle +   \langle  \nabla \mathcal{L}(x_j,y_j) ,n_j \rangle + \langle  \widetilde W_j (g_j),\widetilde t_j \rangle.
\end{equation*}

In order for our step to be sufficiently informative and useful, we will impose 
the following conditions on the operator $\widetilde W_j$: 

\begin{subequations} \label{eq:condWjapprox}
\begin{eqnarray}
\label{eq:condWjapproxn} 
\|\widetilde W_j (n_j)  - W_j n_j\| &\le  &\frac{\xi_0}{\gamma^2_j} \\
\label{eq:condWjapproxg} 
\|\widetilde W_j (g_j)  - W_j g_j\| &\le  &\frac{\xi_1}{\gamma_j}  
\end{eqnarray}
\end{subequations}
for some $\xi_0>0$ and $\xi_1>0$.  
It can be shown~\cite[Theorem B.1]{heinkenschloss2014matrixfreeSQP} that 
the use of inexact linear algebra leads to an approximation $\widetilde W_j$ 
of $W_j$ such that~\eqref{eq:condWjapproxn} 
and~\eqref{eq:condWjapproxg} hold. In our experiments, this corresponds 
to applying the \texttt{minres}~\cite{SCTChoi_CCPaige_MASaunders_2011} solver 
to~\eqref{compute_W_j_exactly} with appropriate choices for the 
hyperparameters so as to satisfy the conditions described above.

Similarly to the quasi-normal step, we also require $\widetilde t_j$ to 
satisfy a fraction of the Cauchy decrease condition on the model 
$\widetilde m^l_j$, i.e.
\begin{equation}\label{eq:treq:inexact} 
\widetilde m^l_j(n_j)-\widetilde m^l_j(n_j + \widetilde t_j)
\ge \kappa_2 \frac{\|\widetilde W_j (g_j)\|^2}{\|J^F_j\|^2+\gamma_j}, 
\end{equation}
for some constant \revised{$\kappa_2>0$}.

Note that since we consider the use of inexact steps, the vector 
$\widetilde t_j=\widetilde W_j (w_j)$ might not belong to the null space of $J^c_j$. 
Following previous methodology proposed for matrix-free SQP 
trust region~\cite{heinkenschloss2002analysis,heinkenschloss2014matrixfreeSQP}, we 
compute a step $t_j$ close to the projection of $\widetilde t_j$ onto this 
null space. More precisely, we enforce the following requirement on $t_j$: 
\begin{equation} \label{eq:tminusWtildet} 
\| t_j  - W_j \widetilde t_j\| \le  \frac{\xi_2}{\gamma_j^2}. 
\end{equation}

\subsection{Nonmonotone acceptance rule}
\label{subsec:algo:nonmonotone}

Having computed our inexact steps, we now need to determine whether they are 
sufficiently promising to deserve acceptance. As in standard SQP trust-region 
approaches, we compare the decrease predicted by the model and the actual variation 
produced by the step. Classical monotone frameworks require the actual reduction to be 
larger than a fraction of the predicted reduction, which may impose unnecessarily 
severe restrictions on the step~\cite{MUlbrich_SUlbrich_2003}. We thus adopt a 
nonmonotonic step acceptance procedure detailed below. 

We first define the \emph{predicted reduction} and the \emph{actual reduction}
to the constraint violation function by
\begin{eqnarray*}
\pred^c_j &\triangleq &\phi(x_j)-m^c_j(n_j) = \frac{1}{2} \|C(x_j)\|^2 - m^c_j(n_j), \\
\ared^c_j &\triangleq &\phi(x_j)-\phi(x_j+s_j) 
=\frac{1}{2} \|C(x_j)\|^2 - \frac{1}{2} \|C(x_j+s_j)\|^2.
\end{eqnarray*}
In a monotone framework, the quasi-normal step $n_j$ is accepted if $\ared^c_j$ is 
larger than a fraction of the predicted reduction $\pred^c_j$: our nonmonotonic 
approach requires instead that
$
\rared^c_j\ge \rho_1 \pred^c_j, 
$
where $\rho_1 \in (0,1)$, and $\rared^c_j$ defines the 
\emph{relaxed actual reduction} of $\phi$, i.e.
\begin{equation} \label{eq:rared_c}
\rared^c_j \; \triangleq \;  \frac{1}{2} \max \left \{ R_j, 
\sum^{\nu^c_j -1}_{k=0} \mu^c_{jk} \|C_{j-k}\|^2 \right \}
- \frac{1}{2} \|C(x_j+s_j)\|^2,
\end{equation}
where $\nu \in \mathbb{N}^*$, $\mu\in(0,1/m)$, and the quantities $R_j$, $\mu_{jk}^c$, 
$\nu^c_j$ satisfy
\begin{equation*}\label{eq:params_filter}
\nu^c_j \triangleq  \min(j+1,\nu), \, 
\mu^c_{jk}\ge \mu>0,\, 
\sum^{\nu^c_j -1}_{k=0} \mu^c_{jk}=1, \,R_j\ge \|C_j\|^2
\end{equation*}

In order to compute the quantity $R_j$, we rely on an auxiliary 
procedure described in Algorithm~\ref{alg:R}. If 
the constraint violation is not significantly smaller than $\|\hat{g}_j\|$, where 
$\hat{g}_j = \widetilde W_j (g_j)$ is the reduced gradient, then $R_j$ is set to 
$\|C_j\|^2$ so as to give preference to steps that improve feasibility. On the 
other hand, if the constraint violation is much smaller than the norm of the reduced 
gradient, $R_j$ is set to a value larger than $\|C_j\|^2$ but smaller than a given 
upper bound $a_{k_j}$ where \revised{$\{a_k\}$}     is a slowly decreasing sequence such that
$
a_k>0,\,0<\alpha_0\le \frac{a_{k+1}}{a_k} <1, \, \lim_{k\to \infty} a_k=0,~\mbox{and}
 \,\sum_{k=0}^\infty a^\eta_k=\infty$,
where $4>\eta>4/3$ is a fixed constant (see~\cite{MUlbrich_SUlbrich_2003} for 
details on this procedure).
\begin{algorithm}[ht!]
\caption{\bf \bf Updating procedure for $R_j$}
\label{alg:R}
\begin{algorithmic}[1]
\REQUIRE $\alpha>0$, $\beta<1/2$, $k_j$ and \revised{$\{a_k\}$}.
\smallskip
\hrule
\smallskip
\IF{$\|C_j\|<\min\{\alpha a_{k_j},\beta\|\hat{g}_j\|\}$}
\STATE $R_j=\min\{a^2_{k_j},\|\hat{g}_j\|^2\}$.
\IF{$R_j\ge \sum^{\nu^c_j -1}_{k=0} \mu^c_{jk} \|C_{j-k}\|^2$} 
\STATE Set $k_{j+1}=k_j+1$.
\ELSE
\STATE Set $k_{j+1}=k_j$.
\ENDIF
\ELSE
\STATE Set $R_j=\|C_j\|^2$ and $k_{j+1}=k_j$.
\ENDIF
\end{algorithmic}
\end{algorithm}

As for the quasi-normal steps, we now introduce a nonmonotone acceptance rule 
for the tangential step. 
The predicted reduction for the step $\widetilde s_j = n_j + \tilde{t}_j$ is 
computed as:
\begin{equation*}
\pred^t_j 
\triangleq \widetilde m^l_j(n_j) -  \widetilde m^l_j(n_j+\widetilde t_j) 
= - \frac{1}{2} \langle H_j \widetilde t_j, \widetilde t_j  \rangle 
- \langle \widetilde W_j (g_j), \widetilde t_j \rangle,
\end{equation*}
while the predicted reduction for the step $s_j=n_j+t_j$ is defined as
\begin{equation*}
\pred^l_j 
\triangleq \widetilde m^l_j(0) -  \widetilde m^l_j(\widetilde s_j) 
+ \frac{1}{2}\langle \gamma_j t_j+g_j,n_j-\widetilde W_j (n_j)\rangle,
\end{equation*}
The actual reduction of the Lagrangian function $\mathcal{L}$ can be written as
\begin{equation*}
\ared^l_j \triangleq \mathcal{L}(x_j,y_j) - \mathcal{L}(x_j+s_j,y_j)
\end{equation*}
and finally the relaxed (nonmonotone) actual reduction of $\mathcal{L}$ is defined by,
\begin{equation*}
\rared^l_j \triangleq \max \left \{  \mathcal{L}(x_j,y_j), \sum^{\nu^l_j -1}_{k=0} \mu^l_{jk} \mathcal{L}(x_{j-k},y_{j-k}) \right \} - \mathcal{L}(x_j+s_j,y_j),
\end{equation*}
where, similarly to~\eqref{eq:rared_c}, $\mu$ is chosen as in~\eqref{eq:rared_c}, and
\[
\nu^l_j = \min(j+1,\nu^l), \, \mu^l_{jk}\ge \mu>0,\, \sum^{\nu^l_j -1}_{k=0} \mu^l_{jk}=1.
\]
with $\nu^l \in \mathbb{N}^*$ (Note that for simplicity, we 
may, and do, in our numerical experiments, choose $\nu^l=\nu$).

\revised{
Overall, in our proposed algorithm, two sets of conditions can lead to acceptance of 
the step $s_j$. First, if the step $s_j$ satisfies  
 $\pred_j^t \ge \max\{\pred_j^c, (\pred_j^c)^\xi \}$, $\pred_j^l \ge\rho_2 \pred_j^t$, $\rared_j^l \ge \rho_1 \pred_j^l$, and $\rared_j^c \ge \rho_1 \pred_j^c$ (where $\xi$, $\rho_2$ and $\rho_1$ are pre-specified constants), then 
this step can improve both optimality and feasibility (in a nonmonotone sense), and we thus accept it. Secondly, if $s_k$ 
satisfies  $\pred_j^t < \max\{\pred_j^c, (\pred_j^c)^\xi \}$, $\pred_j^l < \rho_2 \pred_j^t$ and $\rared^c_j \ge \rho_1\pred^c_j$, we accept this step and focus on improving feasibility (in a nonmonotone 
sense).}



\subsection{Main algorithm}
\label{subsec:algo:fullalgo}

A formal description of the complete algorithm is given in 
Algorithm~\ref{alg:LMeq}. Note that it encompasses both exact and inexact 
variants of our method.
Note also that we do not need to specify a procedure to compute the Lagrange 
multiplier estimate, as those do not play a major role in our global 
convergence theory. One standard choice, that we adopted in our numerical 
experiments, is the least-squares multipliers, i.e. the solution to 
$\min_y \|g_j-J^c_j y\|^2_2$ (note that this subproblem is another 
unconstrained linear least-squares problem).

\begin{algorithm}[ht!]
\caption{\bf \bf A nonmonotone matrix-free LM for equality constraints.}
\label{alg:LMeq}
\begin{algorithmic}[1]
\REQUIRE $\rho_1,\rho_2 \in (0,1)$, $0<\hat\gamma_1<1<\hat\gamma_2$, 
$0<\alpha,\beta<1/2$, $2/3<\xi<1$ $0<\hat\gamma<1$, $\alpha_0\in(0,1)$, and 
a sequence \revised{$\{a_k\}$}. $k_0 = 0$. 
\smallskip
\hrule
\smallskip
\STATE Choose an initial $x_0$ and $\gamma_j>0$.
\FOR{$j= 0,  1, \ldots$}
\STATE \textbf{Step 1:} Evaluate $F_j$, $J^F_j$, $C_j$, $J^c_j$, $g_j$ and 
a Lagrange multiplier estimate $y_j$.
\STATE \textbf{Step 2:}  Choose \smash{$\{\mu^c_{jr}\}$} and \smash{$\{\mu^l_{jr}\}$}, 
then update $R_j$ using Algorithm~\ref{alg:R}.
\STATE \textbf{Step 3:}  Compute $n_j$ such that condition (\ref{eq:decrease:c}) holds and $\widetilde t_j$ satisfying the conditions~\eqref{eq:condWjapproxn}, \eqref{eq:condWjapproxg}, and \eqref{eq:treq:inexact}. Set $\widetilde s_j=n_j+\widetilde t_j$.
\STATE \textbf{Step 4:}  Compute $t_j$ satisfying  the condition (\ref{eq:tminusWtildet})  and set $ s_j=n_j+ t_j$\;
\STATE \textbf{Step 5:} 
\IF{$\pred^t_j\ge \max\{\pred^c_j,(\pred^c_j)^\xi\}\text{ and }  \pred^l_j\ge \rho_2 \pred^t_j$}
\IF{$\rared^c_j\ge \rho_1 \pred^c_j$ and $\rared^l_j\ge \rho_1 \pred^l_j$}
\STATE Set $\gamma_j=\max(\gamma_{min},\hat\gamma_1\gamma_j)$ and accept the step,i.e., $x_{j+1}=x_j+s_j$.
\ELSE
\STATE Set $\gamma_j=\hat\gamma_2\gamma_j$ and go to \textbf{Step 4}.
\ENDIF
\ELSE 
\IF{$\rared^c_j\ge \rho_1 \pred^c_j$}
\STATE Set $\gamma_j=\max(\gamma_{min},\hat\gamma_1\gamma_j)$ and accept the step,i.e., $x_{j+1}=x_j+s_j$.\ELSE
\STATE Set $\gamma_j=\hat\gamma_2\gamma_j$ and go to \textbf{Step 4}.
\ENDIF
\ENDIF
\ENDFOR
\end{algorithmic}
\end{algorithm}


\section{Global convergence}
\label{sec:globcv}

\subsection{Assumptions and intermediary results}
\label{subsec:globcv:assum}

We will establish global convergence of Algorithm~\ref{alg:LMeq} under the 
following standard set of assumptions.

\begin{assumption} \label{assum:compactxj}
The sequence $\{x_j,x_j+s_j\}$ lies in a compact set $\Omega$.
\end{assumption}

\begin{assumption} \label{assum:smoothfphi}
\revised{
The functions $F$ and $C$ are continuously differentiable (thus 
$\phi$ is too). In addition, the gradients of the functions $f$, $\phi$ 
and $C_i$ are Lipschitz continuous.}
\end{assumption}

Though the rest of the paper, $L^f$ and $L_{\phi}$ will denote 
Lipschitz constants for the gradients of $f$ and $\phi$, respectively.
Note that Assumption~\ref{assum:smoothfphi} implies that the constraint
Jacobian $J^c(\cdot)$ is also Lipschitz continuous: through the rest of the  
paper, $L^c$ will be used as the Lipschitz constant for this Jacobian matrix.

Assumptions~\ref{assum:compactxj} and~\ref{assum:smoothfphi} imply 
that the functions $F$, $C$, $f$, $\phi$ and their derivatives are bounded. 
In  what follows, we will make use of constants $\kappa^f,\kappa^\phi, 
\kappa^f_g,\kappa^\phi_g,\kappa^F_J,\kappa^c_J$ such that 
for any $x \in \Omega$, we have
\begin{subequations} \label{eq:boundsFCfphi}
\begin{align}
\label{eq:boundsfphi0}
f(x) \le \kappa^f,~~\qquad \phi(x) &\le \kappa^\phi, \\
\label{eq:boundsfphi1}
\|\nabla f(x)\| \le \kappa^f_g,\quad \|\nabla \phi(x)\| &\le \kappa^\phi_g, \\
\label{eq:boundsFC1}
\|J^F(x)\| \le \kappa^F_J,\quad \|J^c(x)\| &\le \kappa^c_J.
\end{align}
\end{subequations}

We will add the following assumption to the above properties.

\begin{assumption} \label{assum:boundsJJW}
There exists $\kappa^c_{JJ}>0$ and $\kappa_W>0$ such that
such that for every index $j$, we have
$\left\| \left(J^c_j {J^c_j}^\top\right)^{-1} \right\| \le \kappa^c_{JJ}$
and $\|\widetilde W_j(x)\| \le \kappa_W \|x\|$ for any $x$.  
\end{assumption}

\begin{assumption} \label{assum:yj}
There exists $\kappa_y>0$ such that 
$\|y_j\| \le \kappa_y$ for every $j$.
\end{assumption}

Equipped with these assumptions, we can now state and prove several 
bounds on algorithmic quantities. To this end, we first state  
properties of the quasi-normal and tangential steps in 
Lemma~\ref{lem:stepbounds}. Note that those arise from the analysis 
of the two unconstrained problems~\eqref{normalstep} and~\eqref{tangentstep2}, 
and apply to exact as well as inexact solutions of these subproblems 
(see~\cite[Lemma 2.1]{EBergou_YDiouane_VKungurtsev_2017} 
and~\cite[Lemma 5.1]{EBergou_SGratton_LNVicente_2016} for details).

\begin{lemma}\label{lem:stepbounds} 
Under Assumptions~\ref{assum:compactxj} to~\ref{assum:yj}, 
for all $j$, one has:
\begin{subequations} \label{eq:nbounds}
\begin{eqnarray}
\label{eq:nbound2}
\|n_j\| &\le &\frac{\|{J^c_j}^{\top} C_j\|}{\gamma_j}, \\
\label{eq:nbound3}
\|\gamma_j n_j+{J^c_j}^{\top} C_j\| &\le &\frac{\|J^c_j\|^2\|{J^c_j}^{\top} C_j\|}{\gamma_j}, 
\end{eqnarray}
\end{subequations}
and
\begin{subequations} \label{eq:tbounds}
\begin{eqnarray}
\label{eq:tbound1}
\| \widetilde t_j \| &\le &\frac{\|\widetilde W_j (g_j)\|}{\gamma_j}, \\
\label{eq:tbound2}
\| \gamma_j \widetilde t_j+  \widetilde W_j (g_j)\| 
&\le &\frac{\|J^F_j\|^2\|\widetilde W_j (g_j)\|}{\gamma_j}.
\end{eqnarray}
\end{subequations}
\end{lemma}

We can then prove the following series of bounds.

\begin{lemma}\label{lem:usefulbounds}
Under Assumptions~\ref{assum:compactxj} to~\ref{assum:yj}, 
the sequences $\left \{\left\|g_j\right\| \right\}$,
$\left\{ \gamma_j \left\| n_j\right\|\right\}$, 
$\left\{\gamma_j \left\|\widetilde t_j\right\|\right\}$, 
$\left\{\gamma_j \left\| t_j\right\| \right\}$, 
$\left\{\gamma_j \left\|\widetilde s_j\right\|\right\}$,  
$\left\{\gamma_j \left\|s_j\right\|\right\}$ and 
$\left\{\gamma_j \| \gamma_j \widetilde t_j+ \widetilde W_j (g_j)\| \right\}$ are 
uniformly bounded from above by a positive constant $b_0>0$.
\end{lemma}
\begin{proof}
Since $g_j={J_j^F}^\top F_j + {J^c_j}^\top y_j + {J^F_j}^\top J^F_j n_j + \gamma_j n_j$, we have:
\begin{equation*}
\|g_j\| \; \le \; \| {J_j^F}^{\top}F_j\| + \|{J^c_j}\| \|y_j\|+ \| {J_j^F}^{\top}J_j^F+\gamma I_d\| \| n_j\|.
\end{equation*}
Using the bounds~\eqref{eq:boundsFCfphi} and~\eqref{eq:nbound2},  we obtain:
\begin{eqnarray*}
\|g_j\| & \le & \kappa_g^f+ \kappa^c_J \kappa_y+ {\kappa^F_J}^2 \frac{\|{J_j^c}^{\top}C_j\| }{\gamma_j} +  \|{J_j^c}^{\top}C_j\|  \le   \kappa_g^f+
 \kappa^c_J\kappa_y+ \frac{{\kappa^F_J}^2 \kappa_g^\phi}{\gamma_{min}}+  \kappa_g^\phi 
 \triangleq a_0.
\end{eqnarray*}
Using the bound on $\|g_j\|$ together with~\eqref{eq:tbound1}, we obtain:
\begin{equation*}
\| \widetilde t_j \|  \le \frac{\|\widetilde W_j (g_j)\|}{\gamma_j}  
\le \frac{\|\widetilde W_j\| \| g_j\|}{\gamma_j}  
\le\frac{\kappa_W a_0}{\gamma_j}.
\end{equation*}
%
Since $\|W_j\|=1$ as $W_j$ is a projection matrix, 
and~\eqref{eq:tminusWtildet} holds, we also have;
\begin{eqnarray*}
\|t_j\| &\le & \|W_j \widetilde t_j \| + \| t_j- W_j \widetilde t_j\|  
\le  \frac{\kappa_W a_0 }{\gamma_j} + \frac{\xi_2}{\gamma_j^2}    
\le \frac{\kappa_W a_0 +\xi_2\gamma_{\min}^{-1}}{\gamma_j },
\end{eqnarray*}
Meanwhile, property~\eqref{eq:nbound2} guarantees that
$
\|n_j\| \le \frac{\|{J^c_j}^{\top} C_j\|}{\gamma_j}
\le \frac{\kappa_g^\phi}{\gamma_j}
$.

Thanks to the three previous bounds on $\|t_j\|$, $\|\widetilde t_j\|$ and 
$\|n_j\|$, we then obtain 
$\|\widetilde s_j\| \le \|n_j\| + \| \widetilde t_j \| 
\le  \frac{\kappa_g^\phi+\kappa_W a_0}{\gamma_j}$, as well as
\begin{equation}\label{eq:sbound}
\|s_j\| 
\le \|n_j\| + \|t_j \|      
\le \frac{\kappa_g^\phi+\kappa_W^2 a_0 +\xi_2\gamma_{\min}^{-1}}{\gamma_j}.
\end{equation}
Finally, property~\eqref{eq:tbound2} in Lemma~\ref{lem:usefulbounds} gives
\begin{eqnarray*}
\| \gamma_j \widetilde t_j+ \widetilde W_j (g_j)\|  
&\le 
&\frac{\|J^F_j\|^2\|\widetilde W_j (g_j)\|}{\gamma_j}  
\le
\frac{(\kappa^F_J)^2 \kappa_W a_0}{\gamma_j}.
\end{eqnarray*}
Setting $b_0  \triangleq \max\{a_0,\kappa_W a_0,   
\kappa_g^\phi+\kappa_W a_0 + \xi_2 \gamma_{\min}^{-1},(\kappa^F_J)^2 \kappa_W a_0\}$ 
gives the desired result.
\end{proof}

The result of Lemma~\ref{lem:usefulbounds} allows us to bound the difference between actual 
and predicted reductions relatively to the regularization parameter.

\begin{lemma}\label{lem:boundactpred} 
Under Assumptions~\ref{assum:compactxj} to~\ref{assum:yj}, there 
exist positive constants $b_1$ and $b_2$ such that 
for every iteration index $j$, one has:
\begin{subequations} \label{eq:boundactpred}
\begin{eqnarray}
\label{eq:boundactpred:3}
|\ared^c_j-2\pred^c_j| &\le &\frac{b_1}{\gamma_j^2},\\
\label{eq:boundactpred:4}
|\ared^l_j-2\pred^l_j| &\le &\frac{b_2}{\gamma_j^2}.
\end{eqnarray}
\end{subequations}
\end{lemma}
\begin{proof}
To lighten the notation, we will omit the index $j$ in the proof. 
We begin by proving~\eqref{eq:boundactpred:3}. Thanks to Assumption~\ref{assum:smoothfphi} and the 
following first-order Taylor expansion of $\phi(\cdot)=\tfrac{1}{2}\|C(\cdot)\|^2$, we 
have:
\begin{small}
$$
\left|\|C(x+s)\|^2 - \|C(x)\|^2 - 2 C^\top J^c s 
- s^\top {J^c}^\top J^c s \right| \le L^\phi\|s\|^2 + \|J^c s\|^2
\le (L^\phi+{\kappa^c_J}^2)\|s\|^2.
$$
\end{small}
Using this formula, we have:
\begin{eqnarray*}
|\ared^c-2\pred^c| 
&= & \left|\frac{1}{2}\|C\|^2 -\frac{1}{2}\|C(x+s)\|^2  
-2 \left(\frac{1}{2}\|C\|^2 - \frac{1}{2} \|C+J^c n\|^2- \frac{1}{2}\gamma \|n\|^2  \right)\right|\\
&= & \left|\frac{1}{2}\|C\|^2 -\frac{1}{2}\|C(x+s)\|^2 - 2 C^\top J^c n 
-\|J^c n\|^2 - \gamma \|n\|^2 \right|\\
&\le &\left|\frac{1}{2}\|C\|^2 + C^\top J^c s + \frac{1}{2}s^\top {J^c}^\top J^c s 
- \frac{1}{2}\|C(x+s)\|^2 \right| \\
&+ &\left| -C^\top J^c s -\frac{1}{2}s^\top {J^c}^\top J^c s
- 2 C^\top J^c n -\|J^c n\|^2 - \gamma \|n\|^2  \right| \\
&\le &\frac{L^\phi+{\kappa^c_J}^2}{2}\|s\|^2 + \left|-C^\top J^c s-\frac{1}{2}s^\top {J^c}^\top J^c s 
- 2 C^\top J^c n -\|J^c n\|^2 - \gamma \|n\|^2  \right|.
\end{eqnarray*}
Using now the decomposition $s=n+t$ and the fact that $J^c W \widetilde t=0$, we can 
reformulate the second term in the last inequality:
\begin{eqnarray*}
& &-C^\top J^c s -\frac{1}{2}s^\top {J^c}^\top J^c s - 2 C^\top J^c n -\|J^c n\|^2 - \gamma \|n\|^2 \\ 
&= &-C^\top J^c t -3 C^\top J^c n - \frac{3}{2} \|J^c n\|^2 - \frac{1}{2}\|J^c t \|^2 
-\gamma\|n\|^2 \\
&= &-(C+J^c n)^\top J^c t -n^\top (\gamma n + {J^c}^\top C) 
- \frac{3}{2} \|J^c n\|^2 - \frac{1}{2}\|J^c t \|^2  \\
&= &-(C+J^c n)^\top J^c (W\widetilde t -t) -n^\top (\gamma n + {J^c}^\top C) 
-\frac{3}{2}\|J^c n\|^2 - \frac{1}{2}\|J^c t \|^2.
\end{eqnarray*}
Hence, we obtain:
\begin{eqnarray*}
|\ared^c-2\pred^c|     
&\le &\frac{L^\phi+{\kappa^c_J}^2}{2}\|s\|^2 + \|(C+J^c n)^\top J^c\| \|W\widetilde t -t\|
+\left|n^\top (\gamma n + {J^c}^\top C)\right| \\
& &+\frac{3}{2}\|J^c\|^2 \|n\|^2 +\frac{1}{2}\|J^c\|^2 \|t\|^2 \\
&\le &\frac{L^\phi+{\kappa^c_J}^2}{2}\|s\|^2 + \left(\|C\|+\|J^c\| \|n\|\right)\|J^c\|\frac{\xi_2}{\gamma^2}    
+ \|n\|\|\gamma n + {J^c}^\top C\| \\
& &+\frac{3}{2}\|J^c\|^2 \|n\|^2 +\frac{1}{2}\|J^c\|^2 \|t\|^2 \\
&\le &\frac{(L^\phi+{\kappa^c_J}^2) b_0^2}{2\gamma^2}+\|J^c\|\left(\|C\|+\|J^c\|\frac{b_0}{\gamma}\right)
\frac{\xi_2}{\gamma^2}  + \frac{b_0\|J^c\|^2 \|{J^c}^\top C\|}{\gamma^2} 
+ \frac{2 b_0\|J^c\|^2}{\gamma^2} \\
&\le &\left[\frac{(L^\phi+{\kappa^c_J}^2) b_0^2}{2}
+\kappa^c_J\left(\kappa^c + \frac{\kappa^c_J b_0}{\gamma_{\min}}\right)\xi_2 
+ b_0{\kappa^c_J}^2 (\kappa^\phi_g+2)\right]\frac{1}{\gamma^2},
\end{eqnarray*}
where we applied~\eqref{eq:condWjapproxg}, \eqref{eq:nbound3}, 
Lemma~\ref{lem:usefulbounds}, \eqref{eq:boundsFCfphi} and $\gamma\ge \gamma_{\min}$.
Hence~\eqref{eq:boundactpred:3} holds with $b_1=\left[\frac{(L^\phi+{\kappa^c_J}^2) b_0^2}{2}
+\kappa^c_J\left(\kappa^c + \frac{\kappa^c_J b_0}{\gamma_{\min}}\right)\xi_2 
+ b_0{\kappa^c_J}^2 (\kappa^\phi_g+2)\right]$.

We now establish~\eqref{eq:boundactpred:4}. The definition of 
$\pred^l$ gives:
\vspace{-2mm}
\begin{eqnarray*}
\pred^l &= 
&- \frac{1}{2} \langle H n, n  \rangle - \frac{1}{2} \langle H \widetilde t  , \widetilde t  \rangle 
- \langle \nabla_x \mathcal{L}(x,y),n \rangle 
- \langle \widetilde W (g),\widetilde t \rangle +\frac{1}{2}\langle \gamma t+g,n-\widetilde W(n)\rangle \\
&= &\frac{1}{2} \langle H n, n  \rangle - \frac{1}{2} \langle H \widetilde t  , \widetilde t  \rangle 
- \langle \nabla_x \mathcal{L}(x,y)+Hn,n+ W\widetilde t \rangle\\
& &- \langle \widetilde W (g)-W g,\widetilde t \rangle +\frac{1}{2}\langle \gamma t+g,n-\widetilde W(n)\rangle \\
&= &\frac{1}{2}\langle Ht,t \rangle- \frac{1}{2} \langle H \widetilde t  , \widetilde t  \rangle 
- \langle \nabla_x \mathcal{L}(x,y)+Hn,W \widetilde t -t \rangle \nonumber \\
& &-\langle \nabla_x \mathcal{L}(x,y),s \rangle - \frac{1}{2} \langle H s,s \rangle  
-\langle \widetilde W (g)-W g,\widetilde t \rangle  +\frac{1}{2}\langle \gamma t+g,n-\widetilde W(n)\rangle,
\end{eqnarray*}
where we first used the formula $g=\nabla_x \mathcal{L}(x,y)+Hn$ and $W=W^\top$, then 
$s=t+n$. Consequently,
\begin{eqnarray*}
-2\pred^l &= &-\langle Ht,t \rangle + \langle H \widetilde t  , \widetilde t  \rangle 
+ 2 \langle \nabla_x \mathcal{L}(x,y)+Hn,W \widetilde t -t \rangle \nonumber \\
& &+ 2\langle \nabla_x \mathcal{L}(x,y),s \rangle + \langle H s,s \rangle  
+2\langle \widetilde W (g)-W g,\widetilde t \rangle-\langle \gamma t+g,n-\widetilde W(n)\rangle \\
&= &\langle \nabla_x \mathcal{L}(x,y),s \rangle + 2 \langle g,W\widetilde t - t \rangle 
+2 \langle \widetilde W (g) - W g,\widetilde t \rangle - \langle \gamma t+g,n-\widetilde W(n)\rangle \\
& &+\langle \nabla_x \mathcal{L}(x,y),s \rangle -\langle Ht,t \rangle + \langle H \widetilde t  , \widetilde t  \rangle + \langle H s,s \rangle \\
& &+\langle g,s \rangle -\langle Hn,s \rangle -\langle Ht,t \rangle + \langle H \widetilde t  , \widetilde t  \rangle + \langle H s,s \rangle \\
&= &\langle \nabla_x \mathcal{L}(x,y),s \rangle + 2 \langle g,W\widetilde t - t \rangle 
+2 \langle \widetilde W (g) - W g,\widetilde t \rangle - \langle \gamma t+g,n-\widetilde W(n)\rangle \\
& &+\langle g,s \rangle +\langle Ht,s \rangle -\langle Ht,t \rangle + \langle H \widetilde t  , \widetilde t  \rangle.
\end{eqnarray*}
Using $H={J^F}^\top J^F + \gamma I_d$, we obtain:
\begin{eqnarray*}
-2\pred^l &= &\langle \nabla_x \mathcal{L}(x,y),s \rangle + 2 \langle g,W\widetilde t - t \rangle 
+2 \langle \widetilde W (g) - W g,\widetilde t \rangle - \langle \gamma t+g,n-\widetilde W(n)\rangle \\
& &+\langle \gamma t+g,s \rangle +\langle {J^F}^\top J^F t,s \rangle -\langle {J^F}^\top J^F t,t \rangle 
- \langle \gamma t,t \rangle + \langle {J^F}^\top J^F \widetilde t  , \widetilde t  \rangle  
+ \langle \gamma \widetilde t,\widetilde t \rangle\\
&= &\langle \nabla_x \mathcal{L}(x,y),s \rangle + 2 \langle g,W\widetilde t - t \rangle 
+2 \langle \widetilde W (g) - W g,\widetilde t \rangle + \langle \gamma t+g,\widetilde W(n)\rangle \\
& &+\langle g,t \rangle +\langle {J^F}^\top J^F t,n \rangle+ \langle {J^F}^\top J^F \widetilde t  , \widetilde t  \rangle  + \langle \gamma \widetilde t  , \widetilde t  \rangle \\
&= &\langle \nabla_x \mathcal{L}(x,y),s \rangle +  2\langle g,W\widetilde t - t \rangle 
+ \langle \widetilde W (g) - W g,\widetilde t \rangle + \langle \gamma t+g,\widetilde W(n)-W n\rangle \\
& &\langle \gamma t+g,W n\rangle +\langle g,t \rangle -\langle W g,\widetilde t \rangle 
+ \langle \gamma \widetilde t+ \widetilde W(g),\widetilde t \rangle\\
& &+\langle {J^F}^\top J^F t,n \rangle +\langle {J^F}^\top J^F \widetilde t  , \widetilde t  \rangle\\
&= &\langle \nabla_x \mathcal{L}(x,y),s \rangle +  2\langle g,W\widetilde t - t \rangle 
+ \langle \widetilde W (g) - W g,\widetilde t \rangle + \langle \gamma t+g,\widetilde W(n)-W n\rangle \\
& &\langle W(\gamma t+g), n\rangle +\langle g,t- W \widetilde t \rangle 
+ \langle \gamma \widetilde t+ \widetilde W(g),\widetilde t \rangle
+\langle {J^F}^\top J^F t,n \rangle +\langle {J^F}^\top J^F \widetilde t  , \widetilde t  \rangle\\
&= &\langle \nabla_x \mathcal{L}(x,y),s \rangle +  \langle g,W\widetilde t - t \rangle 
+ \langle \widetilde W (g) - W g,\widetilde t \rangle + \langle \gamma t+g,\widetilde W(n)-W n\rangle \\
& &\langle W(\gamma t+g), n\rangle 
+ \langle \gamma \widetilde t+ \widetilde W(g),\widetilde t \rangle
+\langle {J^F}^\top J^F t,n \rangle +\langle {J^F}^\top J^F \widetilde t  , \widetilde t  \rangle,
\end{eqnarray*}
where the line before last is obtained using $W^\top = W$. As a result, 
\begin{eqnarray} 
|\ared^l - 2 \pred^l |
&= &|\mathcal{L}(x) - \mathcal{L}(x+s) - 2 \pred^l | \nonumber\\
&\le &|\mathcal{L}(x) + \langle \nabla_x \mathcal{L}(x,y),s \rangle - \mathcal{L}(x+s)| 
+ \|g\| \|W\widetilde t - t \| \nonumber \\
& &
+ \| \widetilde W (g) - W g\|\|\widetilde t \| + (\gamma \|t\|+\|g\|)\|\widetilde W(n)-W n\| 
\nonumber \\
& &+\|W(\gamma t+g)\| \|n\| + \|\gamma \widetilde t+ \widetilde W(g)\|\|\widetilde t\|.
+\|J^F\|^2 \left( \|t\|\|n\| + \| \widetilde t  \|^2 \right) \nonumber
\end{eqnarray}
By Assumptions~\ref{assum:smoothfphi} and~\ref{assum:yj}, we have
\begin{eqnarray} \label{eq:lipLagrangian}
	\left|\mathcal{L}(x,y) 
	+\langle \nabla_x \mathcal{L}(x,y),s \rangle -  \mathcal{L}(x+s,y) \right|
	&\le &\left| f(x+s)-f(x) -\nabla f(x)^\top s \right| \nonumber\\
	& &+ \|y\| \left\| C(x+s)-C(x)-{J^c}^\top s\right\| \nonumber\\
	\left|\mathcal{L}(x,y) 
	+\langle \nabla_x \mathcal{L}(x,y),s \rangle -  \mathcal{L}(x+s,y) \right|
	&\le &\frac{L^f+\kappa_y L^c}{2}\|s\|^2. 
\end{eqnarray} 
Therefore,
\begin{eqnarray} \label{eq:ared2pred}
|\ared^l - 2 \pred^l | &\le &\frac{L^f + \kappa_y L^c}{2}\|s\|^2+ \|g\| \|W\widetilde t - t \| \nonumber \\
& &+ \| \widetilde W (g) - W g\|\|\widetilde t \| + (\gamma \|t\|+\|g\|)\|\widetilde W(n)-W n\| 
\nonumber \\
& &+\|W(\gamma t+g)\| \|n\| + \|\gamma \widetilde t+ \widetilde W(g)\|\|\widetilde t\|.
+\|J^F\|^2 \left( \|t\|\|n\| + \| \widetilde t  \|^2 \right),
\end{eqnarray}
where the last line comes from~\eqref{eq:lipLagrangian}.
To conclude, we need the result below, that uses the properties $W^2=W=W^\top$ and $\|W\|=1$ of 
the projection matrix $W$. One has: 
\begin{eqnarray} \label{eq:Wgammatplusg}
\| W (\gamma t + g) \| &= &\|\gamma W(t - W\widetilde t) + W(\gamma \widetilde t+ \widetilde W (g)) +W(W g - \widetilde W (g))   \| \nonumber \\
&\le &\gamma\|t - W\widetilde t\|+\|\gamma \widetilde t+ \widetilde W (g)\|+\|W g - \widetilde W (g)\| 
\nonumber \\
\| W (\gamma t + g) \| &\le &\frac{\xi_2}{\gamma}+\frac{b_0}{\gamma}+\frac{\xi_1}{\gamma}.
\end{eqnarray}
where the last line uses~\eqref{eq:tminusWtildet}, Lemma~\ref{lem:usefulbounds} 
and~\eqref{eq:condWjapproxg}.

Using~\eqref{eq:Wgammatplusg} as well as~\eqref{eq:condWjapproxn},\eqref{eq:condWjapproxg},\eqref{eq:tminusWtildet}, the bounds from Lemma~\ref{lem:usefulbounds} and~\eqref{eq:lipLagrangian}, 
we can bound all the terms in~\eqref{eq:ared2pred} and we arrive at:
\begin{eqnarray*}
|\ared^l - 2 \pred^l |
&\le & \frac{(L^f+\kappa_y L^c)b_0^2}{2\gamma^2} + \frac{b_0 \xi_2}{\gamma^2} 
+\frac{b_0\xi_1}{\gamma^3} + \frac{2 b_0 \xi_0}{\gamma^2} \\ 
& &+\frac{2 b_0(\xi_2+b_0+\xi_1)}{\gamma^2} + \frac{b_0^2}{\gamma^2} 
+\frac{2{\kappa^F_J}^2 b_0^2}{\gamma^2}.
\end{eqnarray*}
Using $\gamma \ge \gamma_{\min}$, we obtain $|\ared^l - 2 \pred^l | \le \frac{b_2}{\gamma^2}$ 
with\\
$b_2 =b_0\left(\frac{(L^f+\kappa_y L^c + 6+4\kappa^F_J)b_0}{2} + 3 \xi_2 + \frac{\xi_1+2\gamma_{\min}}{\gamma_{\min}} + 2\xi_0\right)$.
\end{proof}

\subsection{Main convergence results}
\label{subsec:globcv:main}

We now present a global convergence analysis for our framework, that is 
inspired by the analysis of nonmonotone trust-region algorithms without 
penalty function~\cite{MUlbrich_SUlbrich_2003}.

We first establish that, if the method has not converged yet, 
Algorithm~\ref{alg:LMeq} eventually computes and accepts a step for a 
sufficiently large $\gamma_j$. This is the purpose of the 
next lemma, which is similar to~\cite[Lemma 1]{MUlbrich_SUlbrich_2003} but 
adapted to our inexact context.
\begin{lemma}\label{lem:welldefined}
Under Assumptions~\ref{assum:compactxj}-\ref{assum:yj}, let 
$\epsilon>0$, and suppose that the $j$-th 
iterate of Algorithm~\ref{alg:LMeq} is such that 
$\|C_j\|+\|\widetilde W_j (g_j)\|>2\epsilon$. Then, there exists 
$\bar{\gamma}>0$ (depending on $\epsilon$, $\|C_j\|$ and $a_{k_j}$) 
such that the step $s_j$ is accepted whenever $\gamma_j>\bar{\gamma}$. 
\end{lemma}
\begin{proof}
Since $\|C_j\|+\|\widetilde W_j (g_j)\|>2\epsilon$, one of the two 
quantities $\|C_j\|$ and $\|\widetilde W_j (g_j)\|$ must be larger than 
$\epsilon$. We thus consider two cases.

\emph{Case 1}: Suppose that $\|C_j\| > \epsilon$. 
By combining~\eqref{eq:decrease:c} and~\eqref{eq:boundsFC1}, we then have that
$
\pred^c_j  \ge \kappa_1\frac{\epsilon^2}{{\kappa^c_J}^2+\gamma_j},
$
while Lemma~\ref{lem:boundactpred} guarantees that
$|\ared^c_j-2\pred^c_j|\le \frac{b_1}{\gamma_j^2}$. Hence, 
\begin{eqnarray*}
\left|2-\frac{\ared^c_j}{\pred^c_j}\right|
&=& \left|\frac{2\pred^c_j}{\pred^c_j}-\frac{\ared^c_j}{\pred^c_j}\right| 
\le  \frac{b_1}{\kappa_1 \epsilon^2} \frac{{\kappa_J^c}^2+\gamma_j}{\gamma_j^2} 
\rightarrow 0\mbox{~~~~~as~~} \gamma_j\to \infty.
\end{eqnarray*} 
Thus there exists $\bar{\gamma}_1 >0$ such that if 
$\gamma_j\ge \bar{\gamma}_1>0$, then 
$\rared^c_j\ge \ared^c_j \ge \rho_1 \pred^c_j$. If either 
$\pred^t_j < \max\{\pred^c_j,(\pred^c_j)^\xi\}$ or 
$\pred^l_j < \rho_2\pred^t_j$, we know that the step will be accepted. 
Otherwise (i.e. $\pred^t_j\ge \max\{\pred^c_j,(\pred^c_j)^\xi\}$
and $\pred^l_j\ge \rho_2\pred^t_j$): 
\begin{equation*}
	\pred^l_j \ge \rho_2 \pred^t_j \ge \rho_2 \pred^c_j \ge 
	\rho_2 \kappa_1 \frac{\eps^2}{{\kappa^c_J}^2+\gamma_j}.
\end{equation*}
Since by Lemma~\ref{lem:boundactpred}, 
$|\ared^l_j-2 \pred^l_j|\le \frac{b_2}{\gamma_j^2}$, we can use the same 
argument than above to show that there exists $\bar{\gamma}_2$ such that 
$\rared_j^l \ge \ared^l_j \ge \rho_1 \pred^l_j$ for 
$\gamma_j \ge \bar{\gamma}_2$, and thus the step is accepted.

\emph{Case 2}: Suppose now that $\|\widetilde W_j (g_j)\|> \epsilon$. 
By~\eqref{eq:treq:inexact} and~\eqref{eq:boundsFC1}, this implies that 
$\pred^t_j\ge \kappa_2 \frac{\epsilon^2}{{\kappa^F_J}^2+\gamma_j}$. We 
consider two subcases.\\
\emph{Case 2.1:} If $\pred^c_j \ge \pred^t_j$, we have 
$\pred^c_j\ge \kappa_2 \frac{\epsilon^2}{{\kappa^F_J}^2+\gamma_j}$ 
and the same argument than in Case 1 can be employed to guarantee that 
the step is accepted for $\gamma_j\ge \bar{\gamma}_3$ for a certain 
$\bar{\gamma}_3>0$.\\
\emph{Case 2.2:} If $\pred^c_j<\pred^t_j$, then the only condition 
required for step acceptance is that $\rared^c_j \ge \rho_1 \pred^c_j$. 
Defining $\epsilon_{k_j}=\min(\alpha a_{k_j},\beta \epsilon)$ (see 
Algorithm~\ref{alg:R}, we then compare $\epsilon_{k,j}$ and $\|C_j\|$.
If $\|C_j\|>\epsilon_{k_j}$, the reasoning of Case 1 (with $\epsilon_{k,j}$
playing the role of $\epsilon$) guarantees that there exists $\bar{\gamma}_4>0$ 
such that the step is accepted when $\gamma_j > \bar{\gamma}_4$.
On the other hand, if $\|C_j\| \le \epsilon_{k_j}$, we have 
$R_j \ge \min(a_{k_j}^2,\epsilon^2) \ge 4 \epsilon_{k_j}^2$, which 
then gives:
\begin{eqnarray*}
	\rared^c_j \ge \frac{1}{2}R_j-\frac{1}{2}\|C(x_j+s_j)\|^2 
	&\ge &\frac{1}{2}R_j - \frac{1}{2}\|C(x_j)\|^2 -{J^c_j}^\top s_j 
	- \frac{L^c}{2}\|s_j\|^2 \\
	&\ge &\frac{3}{2}\epsilon_{k,j}^2 - \frac{\kappa^c_J b_0}{\gamma_j} 
	- \frac{L^c b_0^2}{2 \gamma_j^2},
\end{eqnarray*}
where the last inequality comes from~\eqref{eq:boundsFC1} and 
Lemma~\ref{lem:usefulbounds}. Thus there exists $\bar{\gamma}_5>0$ 
such that $\rared^c_j \ge \rho_1\epsilon_{k_j}^2$ for $\gamma_j > \bar{\gamma}_5$: since 
$\pred^c_j \le \frac{1}{2} \|C_j\|^2 \le \frac{1}{2} \epsilon_{k_j}^2$ 
by definition, we then have $\rared^c_j \ge \rho_1 \pred^c_j$, and the 
step is accepted. 

Letting $\bar{\gamma}=\max\{\bar{\gamma}_1,\bar{\gamma}_2,
\bar{\gamma}_3,\bar{\gamma}_4,\bar{\gamma}_5\}$ finally leads to the desired 
result.
\end{proof}
The remainder of our analysis relies on several arguments that are identical to the 
trust-region setting, which we restate below (see~\cite[Lemmas 2-5]{MUlbrich_SUlbrich_2003} 
for proofs). The analysis relies on considering the steps that have been accepted: for 
this purpose, the subscript $_{j,a}$ will refer to quantities related to the $j$-th iteration 
at which the step was accepted (e.g. $s_{j,a}$ denotes the accepted step at iteration $j$).
\begin{lemma}\label{lem:decreasenonomonotone}
Under Assumptions~\ref{assum:compactxj}-\ref{assum:yj}, 
suppose that there exists $\hat j\ge 0$ such that for all $j\ge \hat j$, we have
$
\rared^c_{j,a}=\max\left\{\|C_j\|^2, \sum^{\nu^c_j -1}_{k=0} \mu^c_{jk} \|C_{j-k}\|^2 \right \} - \|C_{j+1}\|^2.
$
Then, for all $j\ge \hat j$,
$
\|C_{j+1}\|^2 \le \max_{ \hat j-\nu^c_j<l\le  \hat j} R_l-\rho_1\sum_{r= \hat j}^j \mu^{\min\{j-r,\nu\}} \pred^c_{r,a}.
$
\end{lemma}
\begin{lemma}\label{lem:decreasenonomonotone2}
Under Assumptions~\ref{assum:compactxj} to~\ref{assum:yj}, 
suppose that there exists $\hat j\ge 0$ such that, for all $j\ge \hat j$, we have
$
\rared^l_{j,a}+\mathcal{L}(x_{j+1},y_j)-\mathcal{L}(x_{j+1},y_{j+1})\ge \frac{\rho_1}{2}\pred^l_{j,a}
$.\\
Then, for all $j\ge \hat j$,
$$
|\mathcal{L}(x_{j+1},y_{j+1})|\le \max_{ \hat j-\nu^l_j<l\le  \hat j} \mathcal{L}(x_{l},y_{l})-\frac{\rho_1}{2}\sum_{r= \hat j}^j\mu^{\min(j-r,\nu)}\pred^l_{r,a}.
$$
\end{lemma}
\begin{lemma}\label{lem:cakjbound}
Let Assumptions~\ref{assum:compactxj}-\ref{assum:yj} hold. If, at the $j$-th iteration of Algorithm~\ref{alg:LMeq}, Algorithm~\ref{alg:R} performs the 
update $k_{j+1}=k_j+1$, then
$
\|C_{j'}\|\le \frac{1}{\sqrt{\mu}}a_{k_j}
$
for all $j' \ge j$. 
In particular, for all iterations $j$ with $k_j\ge 1$, one has
$
\|C_j\|\le \frac{a_{k_j}}{\sqrt{\mu}\alpha_0}.
$
\end{lemma}
\begin{lemma}\label{lem:lem5ulb}
Let Assumptions~\ref{assum:compactxj}-\ref{assum:yj} hold. If for infinitely many iterations,
\[
\rared^c_{j,a}\neq \max\left\{\|C_j\|^2,\sum_{r=0}^{v^c_j-1}\mu^c_{jr}\|C_{j-r}\|^2\right\}-\|C_{j+1}\|^2
\]
holds, then $k_j\to\infty$ and $\|C_j\|\to 0$.
\end{lemma}

We now have all the ingredients to establish global 
convergence of our framework. We begin by showing that the 
sequence of iterates is asymptotically feasible.
\begin{theorem}\label{theo:feaslimit}
Under Assumptions~\ref{assum:compactxj}-\ref{assum:yj}, 
if Algorithm \ref{alg:LMeq} does not terminate finitely, then
$\lim_{j\to \infty} \|C_j\|=0.$
\end{theorem}
\begin{proof}
The proof proceeds by contradiction. Suppose that 
$\limsup_{j\to \infty} \|C_j\|>0$. Then, by Lemma~\ref{lem:lem5ulb}, 
there exists $ \hat j$ such that for $j\ge  \hat j$,
\[
\rared^c_{j,a}= \max\left\{\|C_j\|^2,\sum_{r=0}^{v^c_j-1}\mu^c_{jr}\|C_{j-r}\|^2\right\}-\|C_{j+1}\|^2
\]
which implies by Lemma~\ref{lem:decreasenonomonotone}:
\begin{equation*}
\|C_{j+1}\|^2 \le M_{\hat j}-\rho_1\sum_{r=  \hat j}^j \mu^{\min\{j-r,\nu\}}\pred^c_{r,a} ~~\text{where} ~~M_{\hat j}=\max_{  \hat j-\nu^c_j<l\le   \hat j} R_l,
\end{equation*}
Since $\mu \in (0,1)$, we thus get for all $j\ge  \hat j$,
\begin{equation}\label{eq:usedecreaselemma}
\|C_{j+1}\|^2 \le M_{\hat j}-\rho_1\mu^{\nu} \sum_{r=  \hat j}^j \pred^c_{r,a}.
\end{equation}
We first consider the case $\liminf_{j\to\infty} \|C_j\|\neq 0$: in that situation, 
there exists an $\epsilon$ for which $\|C_j\|\ge \epsilon$ for
$j\ge  \hat j$. By~\eqref{eq:decrease:c}, this implies
\begin{equation*} 
\pred^c_{j,a}  \ge \kappa_1 \frac{\epsilon^2}{(\kappa^c_J)^2+\gamma_{j,a}},
\end{equation*}
which, combined with~\eqref{eq:usedecreaselemma}, leads to
\begin{equation}\label{eq:gammasumbound}
\sum_{j= \hat j}^\infty \frac{1}{(\kappa^c_J)^2+\gamma_{j,a}} < \infty 
\; \Rightarrow \; \lim_{j \to \infty}\gamma_{j,a}=+\infty.
\end{equation}
Since $\|C_j\|\ge \epsilon$, similarly to \emph{Case 1} in the proof of Lemma~\ref{lem:welldefined}, we 
can show that for sufficiently large $\gamma_j$, the step must be accepted. This, together with the step 
acceptance rule, guarantees that there exists an upper bound for $\gamma_{j,a}$, which contradicts~\eqref{eq:gammasumbound}. Thus, we must have $\liminf_{j\to\infty} \|C_j\| = 0$.  

Because we assumed that $\limsup_{j\to\infty}\|C_j\|\neq 0$, for any $\epsilon>0$, 
there exists a subsequence $\{\underline j\}$ such that $\|C_{\underline j}\|\ge 2\epsilon$. 
Since we just established $\liminf_{j\to\infty} \|C_j\| = 0$, for each index $\underline j$ of 
that subsequence, there exists $\overline j> \underline j$ such that 
$\|C_{\overline j+1}\|<\epsilon$ and $\|C_j\|\ge \epsilon$ for $j=\underline j,...,\overline j$.
For $j=\underline j,...,\overline j$, it thus holds that
\begin{equation} \label{eq:predcepsCV}
\pred^c_{j,a}(n_{j,a}) \ge \kappa_1 \frac{\epsilon^2}{(\kappa^c_J)^2+\gamma_{j,a}},\,j=\underline j,...,\overline j.
\end{equation}
On the other hand, by~\eqref{eq:usedecreaselemma} and our assumption that 
$\limsup_{j\to\infty}\|C_j\|\neq 0$, we have that
$
\sum_{j=\underline j}^{\overline j} \pred^c_{j,a} \to 0 \text{ for }\underline j\to \infty.
$
Meanwhile, Lemma~\ref{lem:usefulbounds} guarantees that $\|s_{j,a}\|\le \frac{b_0}{\gamma_{j,a}}$, 
so that
\begin{eqnarray*}
\|x_{\overline j+1}-x_{\underline j}\| &\le& \sum_{j=\underline j}^{\overline j} \frac{b_0}{\gamma_{j,a}} 
= b_0 \sum_{j=\underline j}^{\overline j} \frac{(\kappa^c_J)^2+\gamma_{j,a}}{\gamma_{j,a}((\kappa^c_J)^2+\gamma_{j,a})}\\
&\le &\frac{b_0}{\epsilon^2 \kappa_1} \left(\frac{(\kappa^c_J)^2}{\gamma_{\min}}+1 \right)\sum_{j=\underline j}^{\overline j} \pred^c_{j,a}
  \to 0 \text{ for }\underline j\to \infty,
\end{eqnarray*}
where the last inequality comes from~\eqref{eq:predcepsCV}.
By Assumption~\ref{assum:smoothfphi}, we then obtain
\[
\epsilon=2\epsilon-\epsilon\le \|C_{\overline j+1}\|-\|C_{\underline j}\|\le\|C_{\overline j+1}- C_{\underline j}\|\le 2\|J^c_j\|\|\|x_{\overline j+1}-x_{\underline j}\| 
+ L^c\|x_{\overline j+1}-x_{\underline j}\|,
\]
and the right-hand side goes to $0$ as $\underline j\to \infty$. 
We have thus reached a contradiction, from which we conclude 
that $\lim_{j \to \infty} \|C_j\|= 0$.
\end{proof}

We now establish convergence towards a certain form of 
stationarity.

\begin{theorem}\label{theo:glimit}
Under the assumptions of Theorem~\ref{theo:feaslimit}, the 
sequence of iterates of Algorithm~\ref{alg:LMeq} is such
that
$
\liminf_{j\to +\infty} \|\widetilde W_j (g_j)\| = 0.
$
\end{theorem}
\begin{proof}
We again seek a contradiction by assuming that there exists $\epsilon>0$ 
such that $\liminf_{j\to +\infty}\|\widetilde W_j (g_j)\|\ge \epsilon$. 
From Lemma~\ref{lem:welldefined}, we know that, in that case, the trial 
step is always accepted if $\gamma_j$ is sufficiently large. By the 
updating rule for $\gamma_j$, this implies that the sequence $\{\gamma_j\}$ 
is bounded from above, i.e. there exists $\gamma_M>0$ such that 
$\gamma_j<\gamma_M$ for all $j$. From~\eqref{eq:treq:inexact}, we 
then have $\pred^t_{j,a}\ge \kappa_2\frac{\epsilon^2}{(\kappa^f_g)^2+\gamma_M}$:
since $\pred^c_{j,a}\le\|C_{j,a}\|^2\to 0$ thanks to Theorem~\ref{theo:feaslimit}, 
for $j$ sufficiently large, 
we must have eventually $\pred^t_{j,a}>\pred^c_{j,a}$.  Furthermore, using Lemma~\ref{lem:usefulbounds},
\begin{eqnarray*}
& & |\pred^l_{j,a}-\pred^t_{j,a}| 
\le |\widetilde m^l_j(0)-\widetilde m^l_j(n_{j,a})| 
+ \frac{1}{2} \| \gamma_j t_{j,a}+g_{j,a}\| \|n_{j,a}-\widetilde W_{j,a} (n_{j,a}) \|
\\ 
&\le &\frac{1}{2}\left(\|J^F_j\|^2+\gamma_{j,a}\right)\|n_{j,a}\|^2
+\left(\|{J^f_{j,a}}^\top F_{j,a}\|+
\|J^c_{j,a}\|\|y_{j,a}\|\right)\|n_{j,a}\|
+ b_0 (1+\kappa_W) \|n_{j,a}\|\\ 
&\le &\frac{1}{2}((\kappa^f_J)^2+\gamma_M)\|n_{j,a}\|^2
+(\kappa^f_J\kappa^f+\kappa^c_g \kappa_y+b_0(1+\kappa_W))\|n_{j,a}\| \\ 
&\le &\frac{1}{2}((\kappa^f_J)^2+\gamma_M)\left(\frac{\|J^c_{j,a}\|\|C_{j,a}\|}{\gamma_{j,a}}\right)^2+ (\kappa^f_g\kappa^f+\kappa^c_g \kappa_y+b_0(1+\kappa_W))\frac{\|J^c_{j,a}\|\|C_{j,a}\|}{\gamma_{j,a}}\\ 
&\le &\frac{1}{2}((\kappa^f_g)^2+\gamma_M)\left(\frac{\kappa^c_J \|C_{j,a}\|}{\gamma_{\min}}\right)^2+ (\kappa^f_g\kappa^f+\kappa^c_g \kappa_y+b_0(1+\kappa_W))\frac{\kappa^c_J \|C_{j,a}\|}{\gamma_{\min}}.
\end{eqnarray*}
The last right-hand side converges to zero by Theorem~\ref{theo:feaslimit}. Thus, 
for $j$ sufficiently large, we must have $\pred^l_{j,a}\ge \rho_2\pred^t_{j,a}$. Since 
the step is accepted, the conditions in Step 5 of Algorithm~\ref{alg:LMeq} 
ensure that we also have $\rared^l_{j,a}\ge \rho_1 \pred^l_{j,a}$. 

In addition, using
\[
|\mathcal{L}(x_{j+1},y_{j+1})-\mathcal{L}(x_{j+1},y_j)|\le \|y_{j+1}-y_j\|\|C_{j+1}\|
\]
together with the fact that we have just established above, $\rared^l_{j,a}\ge \rho_1 \pred^l_{j,a}$,
the boundedness of $y_j$ as given by Assumption~\ref{assum:yj} and $\|C_j\|\to 0$ as shown in Theorem~\ref{theo:feaslimit}, we have for sufficiently large $j$:
\[
\rared^l_{j,a}+\mathcal{L}(x_{j+1},y_j)-\mathcal{L}(x_{j+1},y_{j+1})\ge \frac{\rho_1}{2} \pred^l_{j,a}.
\]
Therefore, by Lemma~\ref{lem:decreasenonomonotone2},
\begin{eqnarray*}
|\mathcal{L}(x_{j+1},y_{j+1})| &\le 
&\max_{\hat j-\nu^l_j<l\le \hat j} \mathcal{L}(x_{l},y_{l}) 
-\frac{\rho_1}{2}\sum_{r= \hat j}^j\mu^{\min(j-r,\nu^l)}\pred^l_{r,a} \\
&\le 
&\max_{\hat j-\nu^l_j<l\le \hat j} \mathcal{L}(x_{l},y_{l}) 
-\frac{\rho_1\rho_2}{2}\sum_{r= \hat j}^j
\kappa_2\mu^{\nu^l} \frac{\epsilon^2}{{\kappa^F_J}^2+\gamma_M},
\end{eqnarray*}
where we used
$\pred^l_{j,a}\ge \rho_2 \pred^t_{j,a}\ge \kappa_2\frac{\epsilon^2}{(\kappa^f_g)^2+\gamma_M}$.
Since the sequence $\{\mathcal{L}(x_{j},y_{j})\}_{j} $ is bounded on $\Omega$,
we thus obtain $\sum_{j=\hat j}^\infty \kappa_2\frac{\epsilon^2}{(\kappa^f_g)^2+\gamma_M} <\infty$
and we arrive at a contradiction, from which we conclude that 
$\liminf_{j \to \infty} \|\widetilde W_j (g_j)\| = 0$.
\end{proof}

\revised{ 
To end this section, we point out that it is possible to strengthen the result of Theorem~\ref{theo:glimit} by 
replacing~\eqref{eq:condWjapproxg} with the following condition:
\begin{equation} \label{new:eq:condWjapproxg}
	\|\widetilde W_j(g_j)-W_j g_j \| \; \le \; 
	\xi_1\min\left\{\|\widetilde W_j(g_j)\|,\frac{1}{\gamma_j}\right\}.	
\end{equation}
Combining this condition with
\begin{eqnarray*}
\| W_j g_j\| + \|C_j\| &\le & \| \widetilde W_j(g_j)\| + \|\widetilde W_j(g_j)-W_j g_j \| + \|C_j\| \le (1+\xi_1)  \| \widetilde W_j(g_j)\| + \|C_j\|,
\end{eqnarray*}
the proofs of Theorems \ref{theo:feaslimit} and \ref{theo:glimit} can be readily 
modified so as to establish the stronger result 
$\liminf_{j \to \infty} \left(\| W_j g_j\| + \|C_j\|\right)=0$.
Although both conditions~\eqref{new:eq:condWjapproxg} and~\eqref{eq:condWjapproxg} are trivially satisfied for an exact step, we point out that a given iterative solver that uses \eqref{new:eq:condWjapproxg} instead 
of \eqref{eq:condWjapproxg} to estimate $\widetilde W_j(g_j)$ may need more 
iterations to converge. In our experiments, both conditions lead to a 
similar performance: interestingly, in both cases, the results were comparable 
to enforcing the condition $\|\widetilde W_j(g_j)-W_j g_j \| \le 10^{-12}$. We 
thus settled on the criterion that was least demanding at the iteration level, 
and adopted condition~\eqref{eq:condWjapproxg} in our 
implementation. 
}

\section{Numerical experiments}
\label{sec:numerics}

In this section, we report the results of several experiments performed in 
order to assess the efficiency and the robustness of Algorithm~\ref{alg:LMeq}. 
We implemented all the algorithms as Matlab m-files. Our tests include 
small-scale standard test cases, a challenging nonlinear nonconvex data 
assimilation task, and two large-scale inverse problems with systems governed 
by PDE-based dynamics. Our main goal is to understand the behavior of our 
method on generic least squares problems compared to standard alternatives, 
and to observe how it handles additional challenges such as nonlinearity in 
both the constraints and the objective functions.

\subsection{Implementation details}
\label{subsec:numerics:implement}

Our parameter values follow that previously adopted for matrix-free 
trust region SQP~\cite{heinkenschloss2014matrixfreeSQP} and nonmonotone 
trust-region methods~\cite{MUlbrich_SUlbrich_2003}. We thus set
$
 \nu =\nu^l=5, ~ \mu=10^{-3}, ~ \rho_1=10^{-2}, ~ \rho_{2}=10^{-2} , 
 ~ \hat \gamma_1 =0.9 , ~\hat \gamma_2 =2, $
 $
  \alpha = \beta =0.1,~ \xi=3/4, ~\gamma_{min}=10^{-16}, 
  ~\mbox{and}~ \gamma_{0}=1.
$
For the sequence $\{ a_k\}$, we used 
$a_0= \min \left\{ 0.1 \max (1, \|C_j\|), 
\|\widetilde W_j(g_j)\|+\|C_j\|\right\} 
~\mbox{and}~a_ k= a_0 (k+1)^{-1/2}~~~\forall k\ge 1
$.
In all our variants, the Lagrange multipliers $y_j$ are computed as 
the solution of $\min_y \|g_j-J^c_j y\|^2_2$. 

We implemented Algorithm~\ref{alg:LMeq} in an exact and an inexact variant, 
respectively using direct and iterative linear algebra.
For the exact variant, named \texttt{LM-EC-EXACT}, the 
subproblems~\eqref{normalstep} and~\eqref{tangentstep2} are solved 
with the backslash Matlab operator (which uses the \revised{\texttt{UMFPACK}} Fortran 
library). Direct elimination based on Matlab's LU-factorization routine was 
used to compute $W_j=\Gamma_j (\Gamma_j^{\top} \Gamma_j)^{-1}\Gamma_j^{\top}$ 
where $\Gamma_j = P_j^{\top} \begin{bmatrix} -L_j^{-\top} N_{j}^{\top} 
\\ I_{n-p} \end{bmatrix},$ where $N_j \in \mathbb{R}^{(n-p)\times p}$, 
$L_j \in \mathbb{R}^{p\times p}$ (a lower triangular matrix), and $P_j$ 
(permutation matrix) are computed from a factorization of the Jacobian matrix 
$J^c_j$ of the form $\begin{bmatrix} L_j  \\ N_{j} \end{bmatrix}  
R_j = P_j J^{c}_j$. Note that $W_j$ is computed explicitly, the steps 
$\widetilde t_j$ and $t_j$ (and thus $\widetilde s_j$ and $s_j$) coincide 
for the exact version. 

The inexact variant, named  \texttt{LM-EC-MATRIXFREE}, is a matrix-free 
implementation of Algorithm~\ref{alg:LMeq}, that is similar in spirit to 
matrix-free trust-region 
implementations~\cite{heinkenschloss2014matrixfreeSQP}. However, since 
our method relies on regularization rather than trust-region, we can use 
a standard conjugate gradient (\texttt{cg}) method to solve our subproblems. 
For approximately solving~\eqref{normalstep}, we apply \texttt{cg} until 
either the residual norm drops below  \texttt{min$\{$1e-4,
\revised{max$\{$1e-15,1e-8$\times\xi^0_n\}\}$}}, where $\xi^0_n$ is the norm of the 
residual after one iteration, or a maximum of 1000 iterations has been  reached. 
Similarly, the approximate tangential step~\eqref{tangentstep_approx} 
is computed using \texttt{cg} with a tolerance of \texttt{min$\{$1e-4,
max$\{$1e-15,\revised{1e-8$\times\xi^0_t\}\}$}} (where $\xi^0_t$ is the norm of the 
residual after one iteration) and a maximum of 1000 iterations; 
to compute the residual vector in an iteration of \texttt{cg}, the 
\texttt{minres} solver~\cite{SCTChoi_CCPaige_MASaunders_2011} is applied 
to~\eqref{compute_W_j_exactly} with the same tolerance. Finally, the vector 
$t_j$ and the projection operator $\widetilde W_j$ are computed using 
\texttt{minres} with the tolerance \texttt{min$\{$1e-4,
max$\{$\revised{1e-15},min$\{\|n_j\|,1/\gamma_j^2\}\}\}$}. 

For comparison, we implemented a Gauss-Newton solver that does not rely on 
regularization: the underlying iteration is $x_{j+1}=x_j + n_j + t_j$, where 
$n_j$ and $t_j$ are respectively solutions of (\ref{normalstep}) and 
(\ref{tangentstep2}) with $\gamma_j=0$. As for our proposed solver, we 
implemented two variants of the Gauss-Newton method; an exact version, termed 
\texttt{GN-EC-EXACT}, where the two steps $n_j$ and $t_j$ are computed using 
the backslash Matlab operator; and a matrix-free implementation, 
named \texttt{GN-EC-MATRIXFREE}, where the steps are computed using the same 
procedure as our matrix-free implementation of Algorithm~\ref{alg:LMeq}.
Our tests also include an implementation of the nonmonotone SQP trust-region 
method~\cite{MUlbrich_SUlbrich_2003}: this trust-region solver will be 
referred to as~\texttt{TR-EC-EXACT}. For the latter solver, we set the initial trust-region radius to 1 and used the 
default setting for the rest of the parameters as in Ulbrich and Ulbrich~\cite{MUlbrich_SUlbrich_2003}. All the subproblems 
in the nonmonotone trust-region \texttt{TR-EC-EXACT} are solved using the 
Steihaug-conjugate gradient method~\cite{TSteihaug_1983}. Unlike our proposed 
algorithm, the nonmonotone trust-region requires an approximate Hessian $B$   
for the objective function $f$: we employ the Gauss-Newton approximation,
 i.e.,  for a given $x \in \mathbb{R}^d$, we set $
B(x)=J^F(x)^{\top}J^F(x)$. Note that the \texttt{TR-EC-EXACT} solver does not 
have a matrix-free counterpart and thus can not be used to solve our most 
challenging PDE inverse problems. For that reason, \texttt{TR-EC-EXACT} will only be 
tested on problems for which it is possible to store the Jacobians of $F$ and 
$C$ in memory.

\subsection{Test on standard least-squares problems}
\label{subsec:numerics:otherls}

In this section, we report numerical results on a \revised{collection 
$\mathcal{P}$ of 20 problems used in Li et al~\cite{ZFLi_MROsborne_TPrvan_2002}, 
formed by the nonlinear least-squares 
problems subject to equality constraints within well-known 
benchmarks~\cite{WHock_KSchittkowski_1980,Schittkowski_1987}: the problem dimensions 
range} between $2$ and $9$. The indexes of the chosen problems as given in the 
reference~\cite{ZFLi_MROsborne_TPrvan_2002} are: $6,$ $26,$  $42,$  $ 47,$  
$ 60, $  $65,$  $ 77,$  $ 79,$  $ 216,$  $ 235,$  $ 249,$  $ 252,$  $ 269,$  
$ 316,$  $ 317,$  $ 318,$  $ 322,$  $ 344,$  $ 345$ and $ 373$. For all 
problems, we used the starting points $x_0$ given in the above reference. A 
method was considered successful if it reached an iterate such that 
$\max(\|C_j\|, \|\widetilde W_j(g_j)\|) \le \revised{10^{-6}}$: if this was not the 
case after $j_{\max}:=1000$ iterations, the method was considered to have 
failed.

To compare the algorithms in this section, we use performance profiles proposed 
by Dolan and Mor\'e~\cite{Dolan_2002}. Given the set of problems $\mathcal{P}$ 
(of cardinality $|\mathcal{P}|$) and a set of algorithms (solvers) 
$\mathcal{S}$, the performance profile $\rho_s(\tau)$ of an algorithm~$s$ is 
defined as the fraction of problems where the performance ratio $r_{p,s}$ is at 
most $\tau$:
\begin{eqnarray*}
 \rho_s(\tau) \; = \; \frac{1}{|\mathcal{P}|} \mbox{size} 
 \{ p \in \mathcal{P}: r_{p,s} \leq \tau \} &~\mbox{where}~~& r_{p,s} 
 \; = \; \frac{t_{p,s} }{\min\{t_{p,s}: s \in \mathcal{S}\}}.
\end{eqnarray*}
The scalar $t_{p,s} > 0$ measures the performance of the algorithm~$s$ when 
solving problem~$p$, seen here as the number of iterations. Better performance 
of the algorithm~$s$ relatively to the other algorithms on the set of problems,
is indicated by higher values of $\rho_s(\tau)$.
In particular, efficiency is measured by $\rho_s(1)$ (the fraction of problems 
for which algorithm~$s$ performs the best) and robustness is measured by 
$\rho_s(\tau)$ for $\tau$ sufficiently large (the fraction of problems solved 
by~$s$). To facilitate the visualization of the results~\cite{Dolan_2002}, we 
plot the performance profiles in a $\log_2$-scale.

\begin{figure}[!ht]
\centering
\subfigure[Exact implementation. \label{fig1:standard:pb}]{
\includegraphics[scale=0.38]{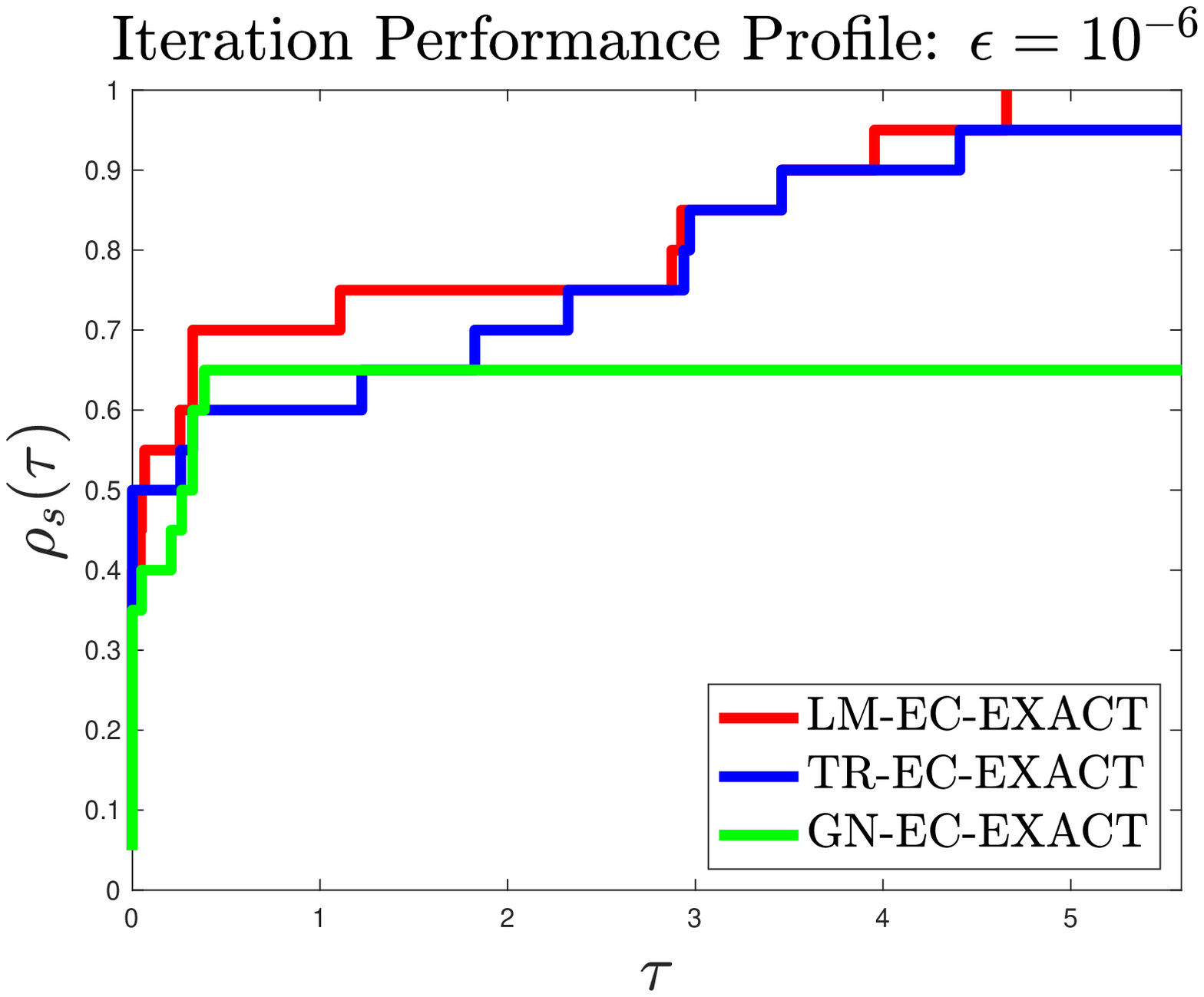}
}
\subfigure[Matrix-free implementation. \label{fig2:standard:pb}]{
\includegraphics[scale=0.38]{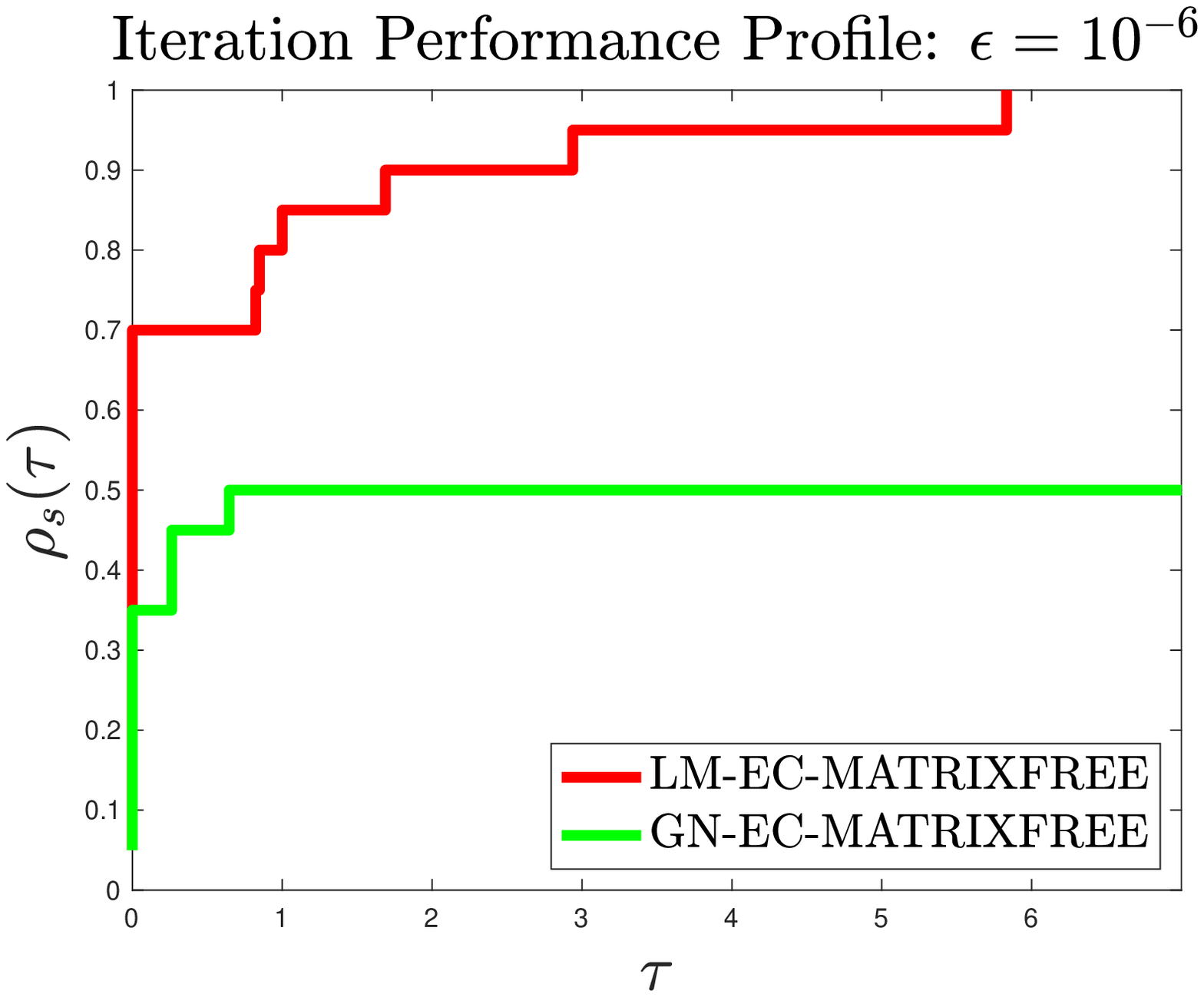}
}
\caption{Iteration performance profiles using $20$ standard constrained 
least-squares problems~\cite{ZFLi_MROsborne_TPrvan_2002}.}
\label{fig::standard:pb}
\end{figure}


\revised{
Figure~\ref{fig::standard:pb} depicts the obtained performance profiles: 
Figure \ref{fig1:standard:pb} compares the exact solvers while 
Figure \ref{fig2:standard:pb} compares the matrix-free variants. In 
both figures, the Gauss Newton methods (\texttt{GN-EC-EXACT} and \texttt{GN-EC-MATRIXFREE})  exhibit the worst performance, 
which one could expect as this variant is only guaranteed to converge 
locally. For instance,  in terms of efficiency,  the Gauss Newton approach is  the 
best only on less than $35\%$ of the problems. In term of robustness (when $\epsilon=10^{-4}$), for the 
exact implementation, \texttt{GN-EC-EXACT} converges for only $65\%$ of 
the problems, whereas \texttt{TR-EC-EXACT} converges for almost all the 
problems and \texttt{LM-EC-EXACT} converges for all problems.} 

\revised{In the 
matrix-free case,  \texttt{GN-EC-MATRIXFREE} converges for $50\%$ of the problems while \texttt{LM-EC-MATRIXFREE} converges for all of the problems. 
Our proposed method and 
\texttt{TR-EC-EXACT} exhibit almost the same performance (with a slight 
advantage for \texttt{LM-EC-EXACT}), which illustrates the efficiency and 
the robustness of globalization approaches (TR and our LM method) compared 
to a basic Gauss-Newton paradigm. Note that these matrix-free variants are 
sensitive to the tolerance for (approximately) solving the subproblems, and 
that fine tuning would likely be necessary to match the performance of their 
exact counterparts. Nevertheless, we will focus on these matrix-free variants 
in Sections~\ref{subsec:numerics:apdecons} and~\ref{subsec:numerics:gwater} as 
the exact variants will not be practical there.
}


\subsection{A data assimilation problem solved using exact linear algebra}  
\label{subsec:numerics:dataassim}

Data assimilation is \revised{concerned with the estimation} of a hidden random 
temporal process $(X_i)_{i=0}^T$, where $X_i$ is the state of the process at
time $i$ and $T$ denotes a time horizon. This \revised{technique} usually combines prior 
information about the process \revised{with} a numerical model and some observations. More 
formally, one aims to determine $x_0, \ldots, x_T$, where $x_i \in \real^n$ is 
an estimator of the state $X_i$, from (i), the prior state $X_0 = x_b + W_b$, 
$W_b \sim N(0,B )$, (ii) the numerical model $ X_i = \mathcal{M}_i(X_{i-1})$, 
$i=1,\ldots,T$, where $\mathcal{M}_i$ is the model operator at time $i$ and 
(iii) the observations  $y_i = \mathcal{H}_i (X_i) + V_i$, $V_i\sim N(0,R_i)$, 
$i=0,\ldots,T$. Here $T$ denotes the time horizon for the assimilation. The 
random vectors $W_b$ and $V_i$ represent the noise on the prior and the 
observation at time $i$, respectively, and are supposed to be Gaussian 
distributed with mean zero and covariance matrices $B$ and $R_i$, respectively. 

The 4DVAR ``strong constraint" method \cite{asch2016} is one of the most 
important data assimilation techniques for weather forecasting, that consists 
in computing $x_0, \ldots, x_T$ by solving the following optimization problem:
{\small{
\begin{equation} \label{function-f}
\begin{array}{rl}
\displaystyle  \min_{(x_0,\ldots,x_T) \in \real^{n(T+1)}} 
& f([x_0,\ldots,x_T]) := \frac{1}{2}\left(\|x_0 - x_b\|^2_{B^{-1}} + 
\sum_{i=0}^{T}\|y_i - \mathcal{H}_i(x_{i})\|^2_{R_i^{-1}}\right)\\
\mbox{s. t.}&   x_i - \mathcal{M}_i(x_{i-1}) = 0, \quad i=0,\ldots, T 
\end{array}
\end{equation}
}}where $x_{-1} = x_b$, $x=(x_0,\ldots,x_T)$, and $\|z\|^2_M=z^\top M z$ is the 
 norm defined by a positive definite matrix $M$. This problem conforms to our 
 generic formulation~\eqref{eq:prb}. In our experiments, the numerical 
model is chosen to be the nonlinear Lorenz 63 system \cite{lorenz63}: for a given 
$x_{i-1}=[x_{i-1}^{(1)}, x_{i-1}^{(2)},x_{i-1}^{(3)}]^{\top} \in \mathbb{R}^3$ 
(i.e., $n=3$), the model is given by
$$  \mathcal{M}_i(x_{i-1}) =\begin{pmatrix} -\sigma (x_{i-1}^{(1)} - x_{i-1}^{(2)}  \\ \rho x_{i-1}^{(1)}-x_{i-1}^{(2)} - x_{i-1}^{(1)} x_{i-1}^{(2)}\\ x_{i-1}^{(1)} x_{i-1}^{(2)} -\beta x_{i-1}^{(3)}
\end{pmatrix},$$
where $\sigma$, $\rho$, and $\beta$ are parameters
whose values are chosen as $10$, $28$, and $8/3$, respectively.
 These values are known to result
in chaotic behavior of the Lorenz 63 model with two regimes  \cite{lorenz63}. 
 We choose the matrices $B$ and $R_i,\, i=0,\ldots,T$ to be identity matrices.  
The observations $\{y_i\}_{i=\{0,\dots,T\}}$ and $x_b$ are generated randomly.
Each variable is observed through the nonlinear operator
 $$\mathcal{H}_i(x_i) = \frac{x_i}{2}\left( 1+ \frac{|x_i|^{\gamma^{\mbox{obs}}-1}}{10}\right),  \quad i=0,\ldots,T,$$
where $|x_i|$ is the component wise absolute value of $x_i$ and $\gamma^{\mbox{obs}}$ is a scalar which tunes the nonlinearity of the observation operator~\cite[Chapter 6]{asch2016}. 
Note that the problem is smooth if $\gamma^{\mbox{obs}}$ is an odd integer. 
 {\small{
   \begin{table}[]
\centering
\caption{Results on the data assimilation problem~\eqref{function-f} using two different $\gamma^{\mbox{obs}}$ values.}
\label{table:da}
\begin{tabular}{|p{.3cm}|p{.48cm}|l|l|l|l|l|l|l|l|l|l|l|l|l}
\hline
\multirow{2}{*}{} &\multirow{2}{*}{$T$} & \multicolumn{4}{c|}{ $\gamma^{\mbox{obs}}=3$} & \multicolumn{4}{c|}{ $\gamma^{\mbox{obs}}=5$}  \\ \cline{3-10} 
 &  & \#it & $f(x_j)$ & $\|C_j\|$  & $\| \hat g_j \|$ & \#it & $f(x_j)$ & $\|C_j\|$ & $\| \hat g_j \|$  \\ \hline \hline 
\multirow{5}{*}{ \rotatebox[origin=c]{90}{{\small{\texttt{LM-EC-EXACT}}}}}    &  2     &  83     &       5.9e+00          &    3.7e-10          &       8.0e-05 &     98     &       6.7e+00          &    7.7e-10          &       8.6e-05     \\ 
   &   3     &  $10^3$     &       8.8e+00          &    1.7e-05          &       1.3e-01    &  $10^3$     &       8.7e+00          &    5.8e-05          &       1.6e-01       \\ 
     &   15     &  70     &       2.6e+01          &    2.4e-09          &       4.5e-05      &     40     &       2.7e+01          &    2.7e-09          &       4.2e-05          \\  
  &45     &  36     &       6.7e+01          &    2.1e-11          &       8.5e-05      &  48     &       6.8e+01          &    1.8e-09          &       7.7e-05        \\  
 &   \small{225}     &  58     &       3.4e+02          &    5.0e-09          &       6.9e-05    &    45     &       3.4e+02          &    1.3e-11          &       8.3e-05          \\ \hline \hline
\multirow{5}{*}{ \rotatebox[origin=c]{90}{{\small{\texttt{TR-EC-EXACT}}}}}    &   2     & $10^3$     &       5.9e+00          &    6.1e-09          &       1.2e-02    &     $10^3$     &       6.6e+00          &    3.6e-08          &       1.5e-02    \\ 
 &  3     & $10^3$     &       7.9e+00          &    2.3e-16          &       1.0e-01     &   $10^3$     &       8.3e+00          &    2.4e-05          &       2.0e-01      \\ 
   &  15     & $10^3$     &       2.6e+01          &    9.1e-06          &       7.0e-02        &$10^3$     &       2.7e+01          &    1.8e-06          &       2.3e-02        \\ 
     &  45     & 73     &       6.7e+01          &    5.7e-09          &       7.6e-05        &     $10^3$     &       6.8e+01          &    1.1e-16          &       2.8e-02         \\  
  & \small{225}     & $10^3$     &       3.4e+02          &    6.7e-16          &       5.8e-03        &  $10^3$     &       3.4e+02          &    2.2e-07          &       5.5e-02              \\ \hline \hline
\multirow{5}{*}{ \rotatebox[origin=c]{90}{{\small{\texttt{GN-EC-EXACT}}}}}    & 2     &  $10^3$     &       3.2e+01          &    1.2e-01          &       1.0e+01           &      $10^3$     &       9.2e+00          &    2.4e-01          &       4.2e+0             \\ 
  &3     &  $10^3$     &       9.7e+00          &    3.2e-02          &       2.5e-01             &    $10^3$     &       9.2e+00          &    6.7e-02          &       9.0e-01     \\ 
   & 15     &  $10^3$     &       9.7e+01          &    2.0e-01          &       2.7e+01       &     $10^3$     &       3.5e+02          &    3.0e-01          &       2.1e+2     \\ 
     &45     &  $10^3$     &       7.4e+01          &    1.2e-01          &       4.7e+00    &    $10^3$     &       1.9e+02          &    4.1e-02          &       7.6e+1        \\  
  & \small{225}     &  $10^3$     &       3.5e+02          &    1.1e-01          &       5.3e+00             &      $10^3$     &       3.5e+02          &    2.7e-02          &       4.7e+0       \\ \hline 
\end{tabular}
\end{table}
}
}
 
 
  \begin{figure}[h]
\centering
\includegraphics[scale=0.34]{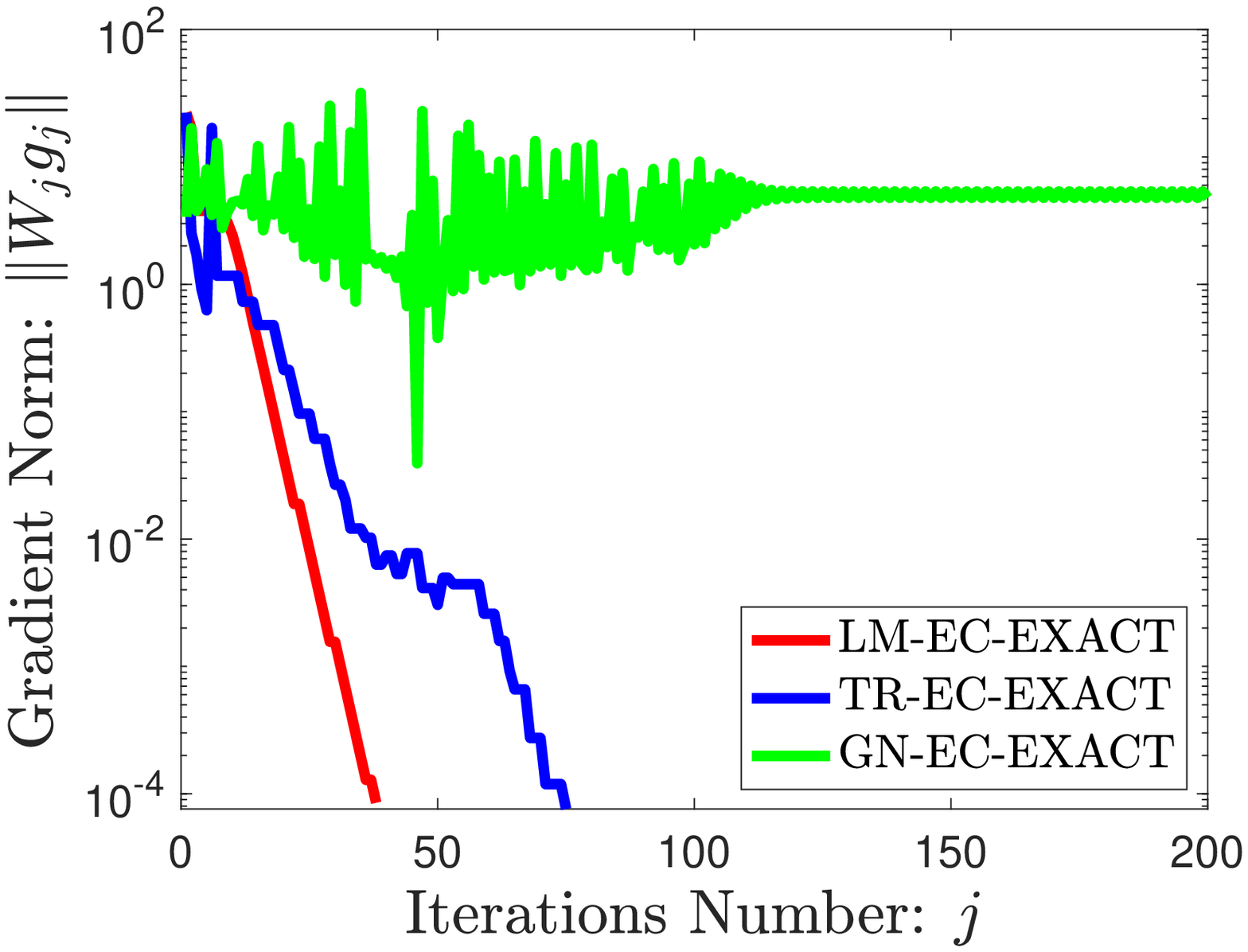}
\includegraphics[scale=0.34]{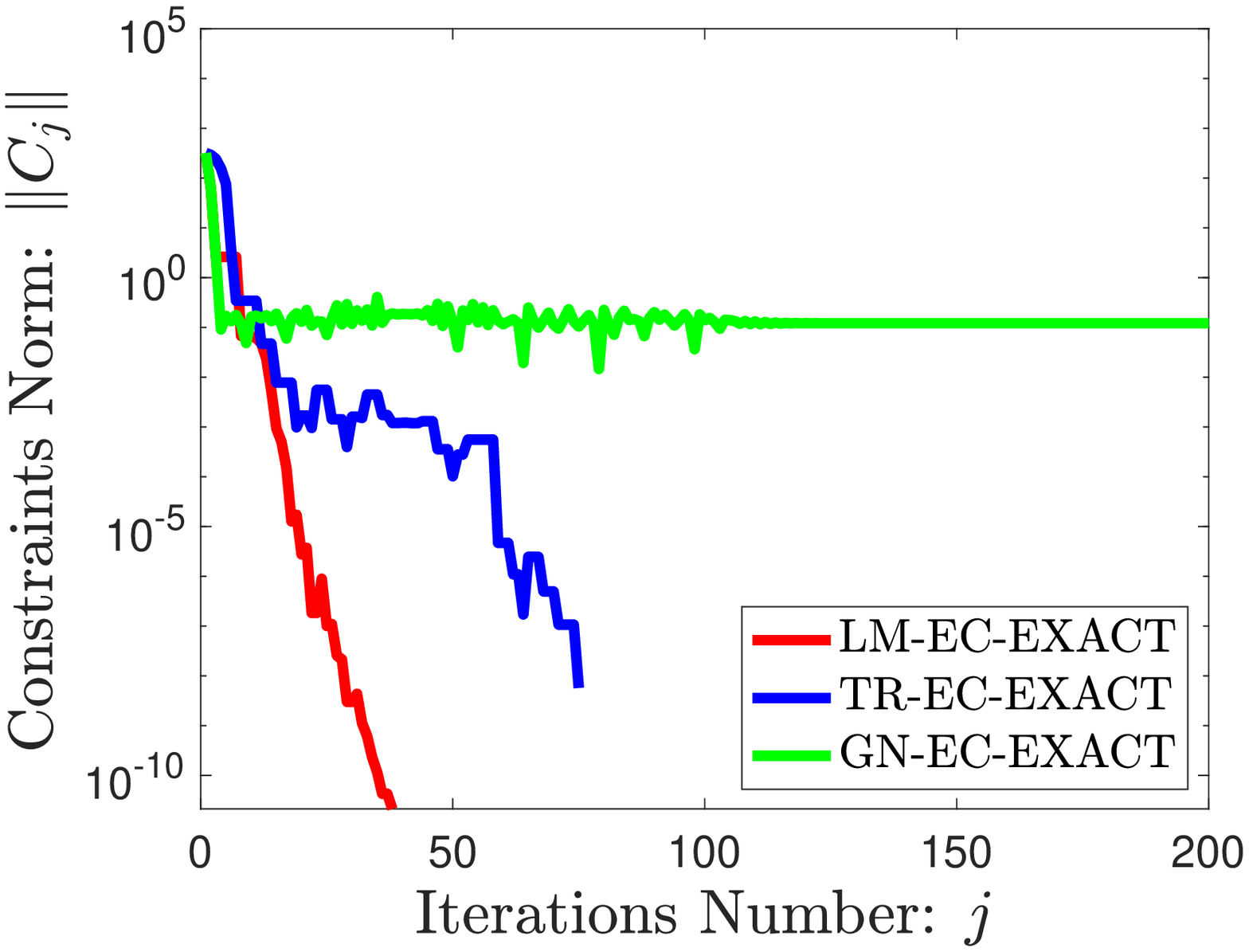}
\caption{Data assimilation convergence plots, considering $T=45$ and $\gamma^{\mbox{obs}}=3$.} \label{fig:ad:obs_op}
\end{figure}

As in the previous section, we tested our three solvers on this problem with 
the convergence criterion $\max(\|C_j\|, \|\widetilde W_j(g_j)\|) \le 10^{-4}$ 
and a maximum of $j_{\max}:=1000$ iterations.
Table \ref{table:da} and Figure~\ref{fig:ad:obs_op} depict the performance of 
the algorithms in terms of the constraint and the projected gradient norms. 
Considering that  $T = 45$ and $\gamma^{\mbox{obs}}=3$,  from 
Figure~\ref{fig:ad:obs_op}, we see that \texttt{LM-EC-EXACT} exhibits better 
performance compared to \texttt{TR-EC-EXACT} and \texttt{GN-EC-EXACT}: the 
latter method even diverges, as the values of the gradient and the constraint 
norm oscillate or stagnate over the iterations. On the contrary, our algorithm 
\texttt{LM-EC-EXACT} and the trust-region method \texttt{TR-EC-EXACT} are able 
to decrease both  the gradient and the constraint norm to a small accuracy, with 
our method converging faster. Table \ref{table:da} confirms the superiority of 
our approach on this problem: \texttt{GN-EC-EXACT} diverges for all instances, 
\texttt{TR-EC-EXACT} shows slightly better results as it converges for some 
instances, and \texttt{LM-EC-EXACT} converges for most of the instances.  

\subsection{A PDE-constrained optimization problem}
\label{subsec:numerics:apdecons}

We now study a least-squares problem with a hyperbolic forward PDE as equality 
constraints~\cite{Haber07modelproblems}. Given a time interval $[0,T]$, a time 
dependent density field $y(x,t)$, and a time velocity field $u(x,t)$, one 
wishes to solve the constrained nonlinear least-squares problem
\begin{equation} \label{function-PDE-opt}
\begin{array}{rl}
\displaystyle  \min_{(y, u)} &   f([y,u]):=\frac{1}{2}\|Q y -z \|^2+ 
 \frac{1}{2} \int_{\Omega} \left((u-u_r)^2 
 + |\nabla(u -u_r)|^2 \right) \\
\mbox{s. t.}& y + \nabla \cdot (yu) =0 \\
& y(0,x) = y_0.
\end{array}
\end{equation}
where $u_r$ is a chosen 
reference model. Given the forward problem on $y$, the operator 
$Q$ represents the projection of $y$  onto the space of the data $z$. Since 
problem~\eqref{function-PDE-opt} is infinite-dimensional, we considered a 
discretization grid such that $d=4096$. To this end, we adapted 
the existing Matlab implementation of this problem available 
online\footnote{http://www.mathcs.emory.edu/$\sim$haber/Code/ModelProblems.tar}.
To initialize the variables in each optimization procedure, we used the 
reference model for $u$ and zero for $y$. Due to the nature of the PDE 
problem, only matrix-free optimization solvers can be used, hence only 
\texttt{LM-EC-MATRIXFREE} and \texttt{GN-EC-MATRIXFREE} were tested on 
this problem. 
\revised{
Figure \ref{fig:pde:op} shows the first $200$ iterations of the two methods.
Within $200$ iterations, our proposed approach is able  to reduce both  the norm of the gradient and the norm of the constraints below $10^{-5}$. The Gauss-Newton method converges to a feasible point but does not reach a first-order optimal solution as the norm of the projected gradient remains large.
}


  \begin{figure}[!h]
\centering
\includegraphics[scale=0.34]{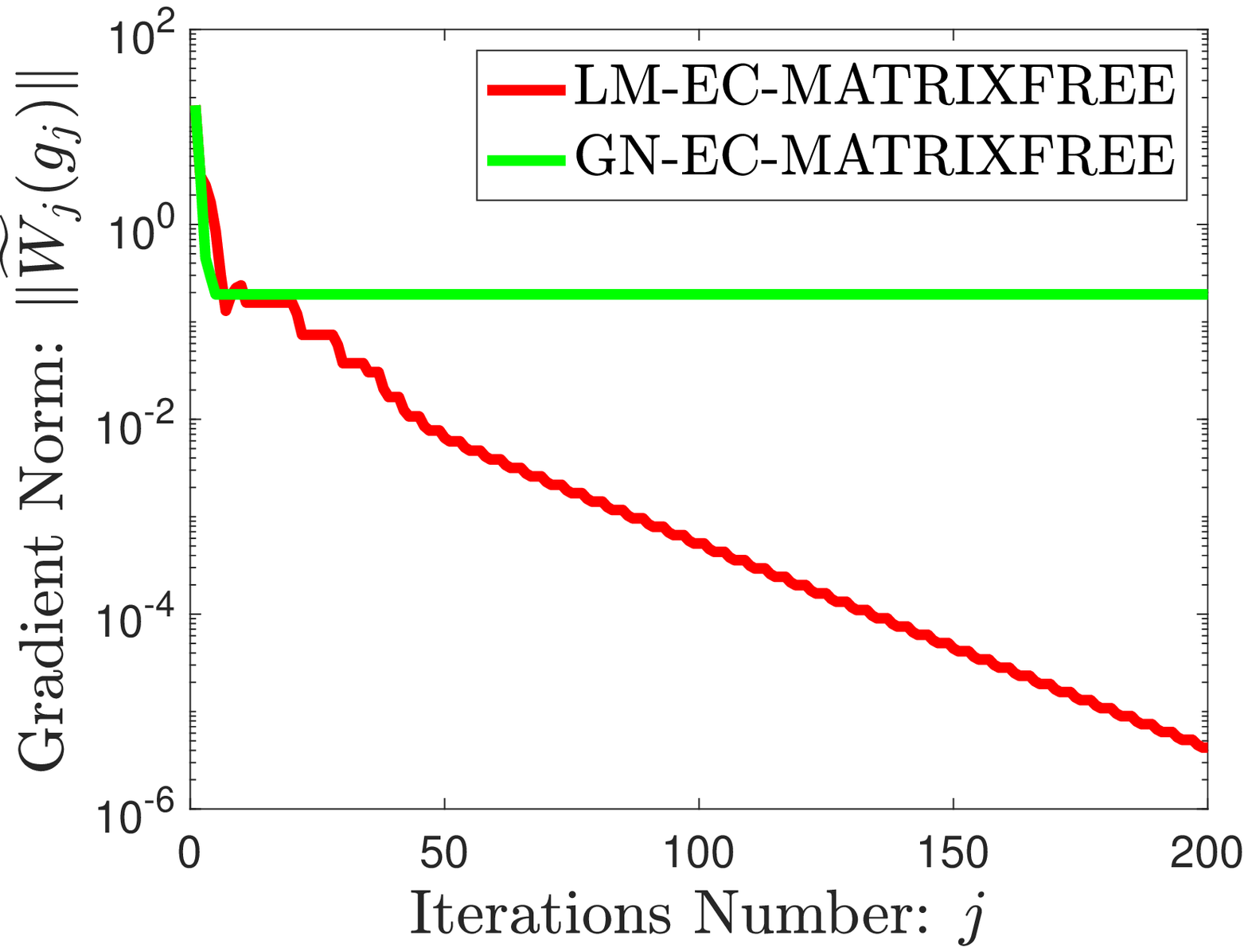}
\includegraphics[scale=0.34]{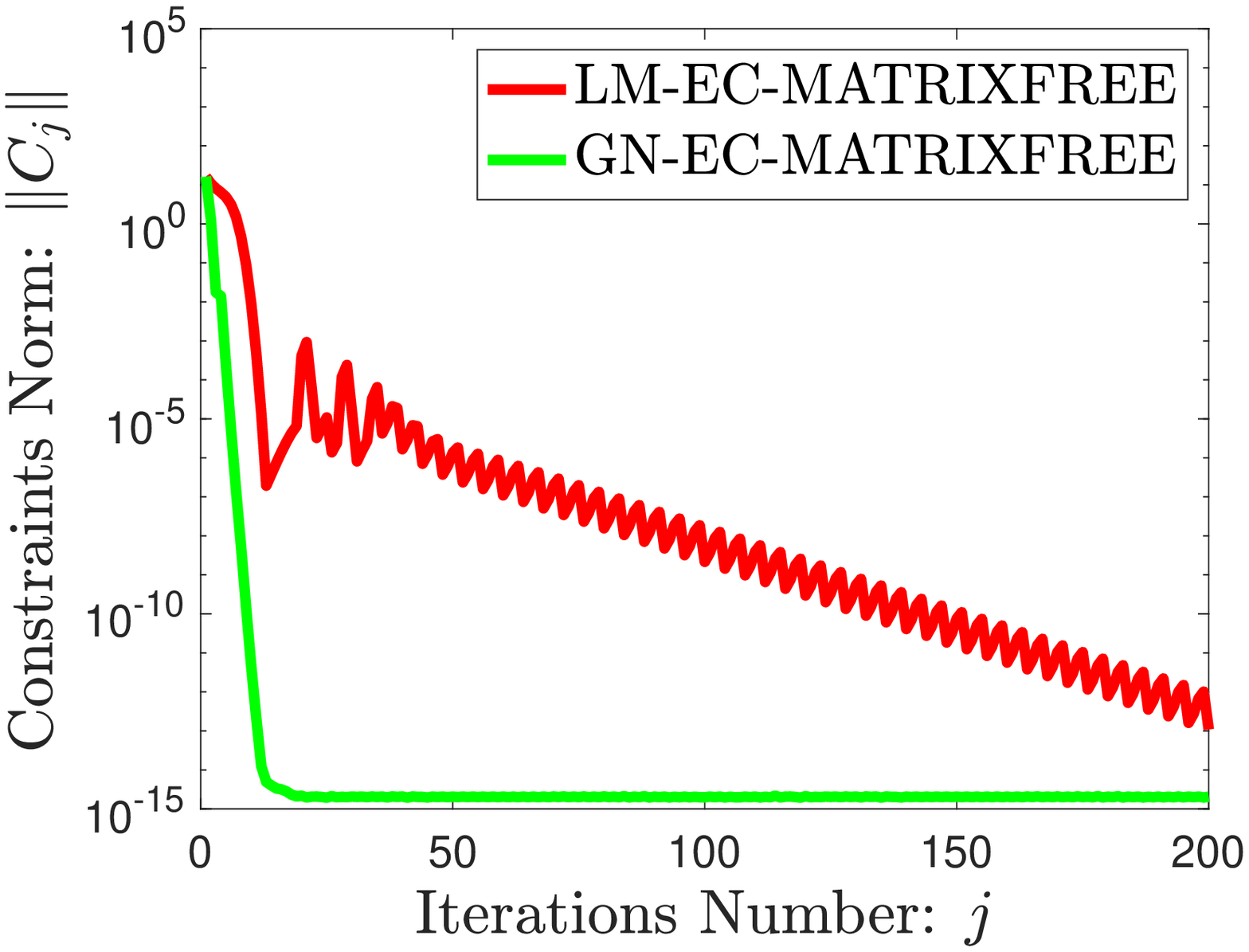}
\includegraphics[scale=0.34]{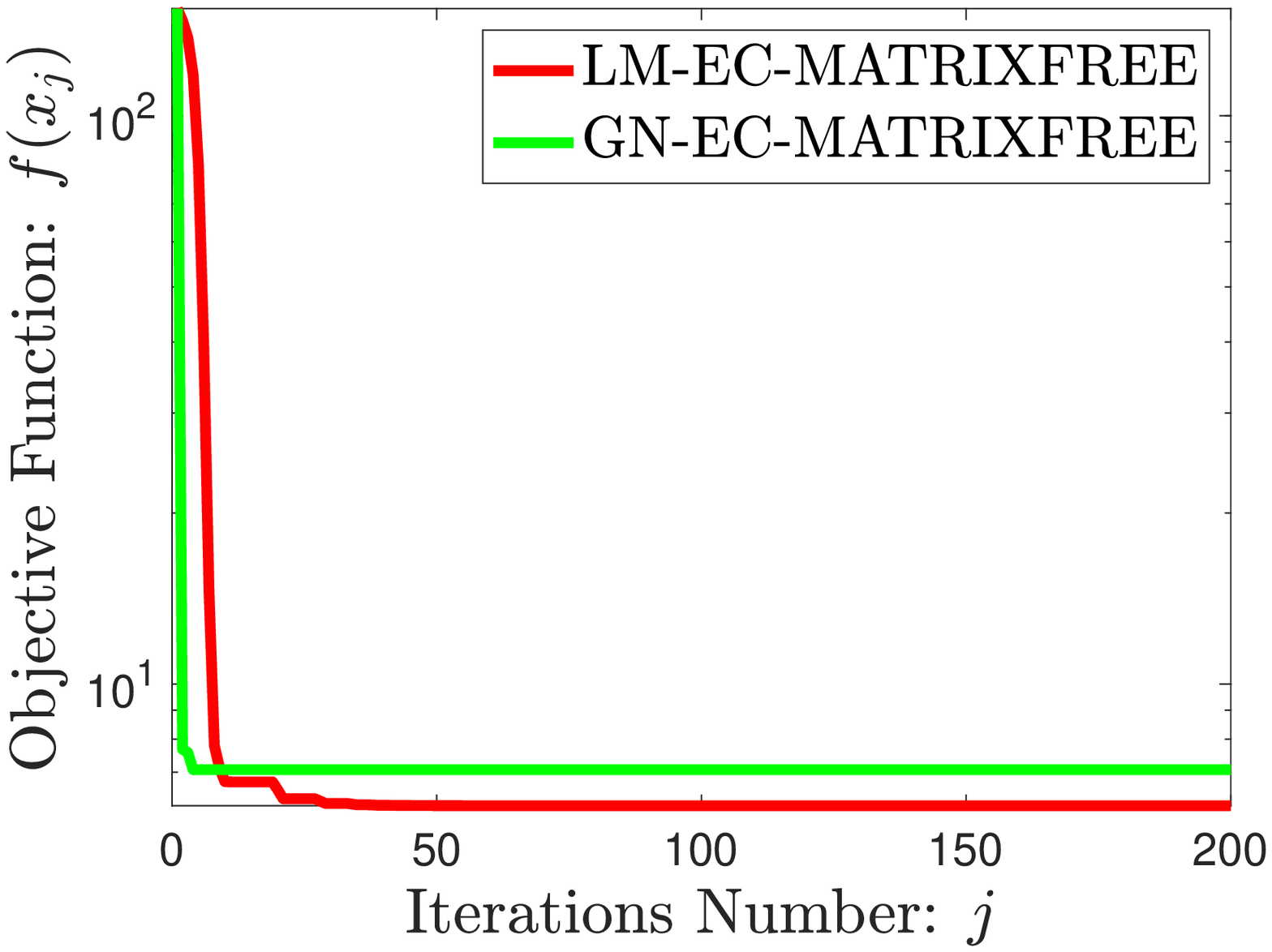}
\includegraphics[scale=0.34]{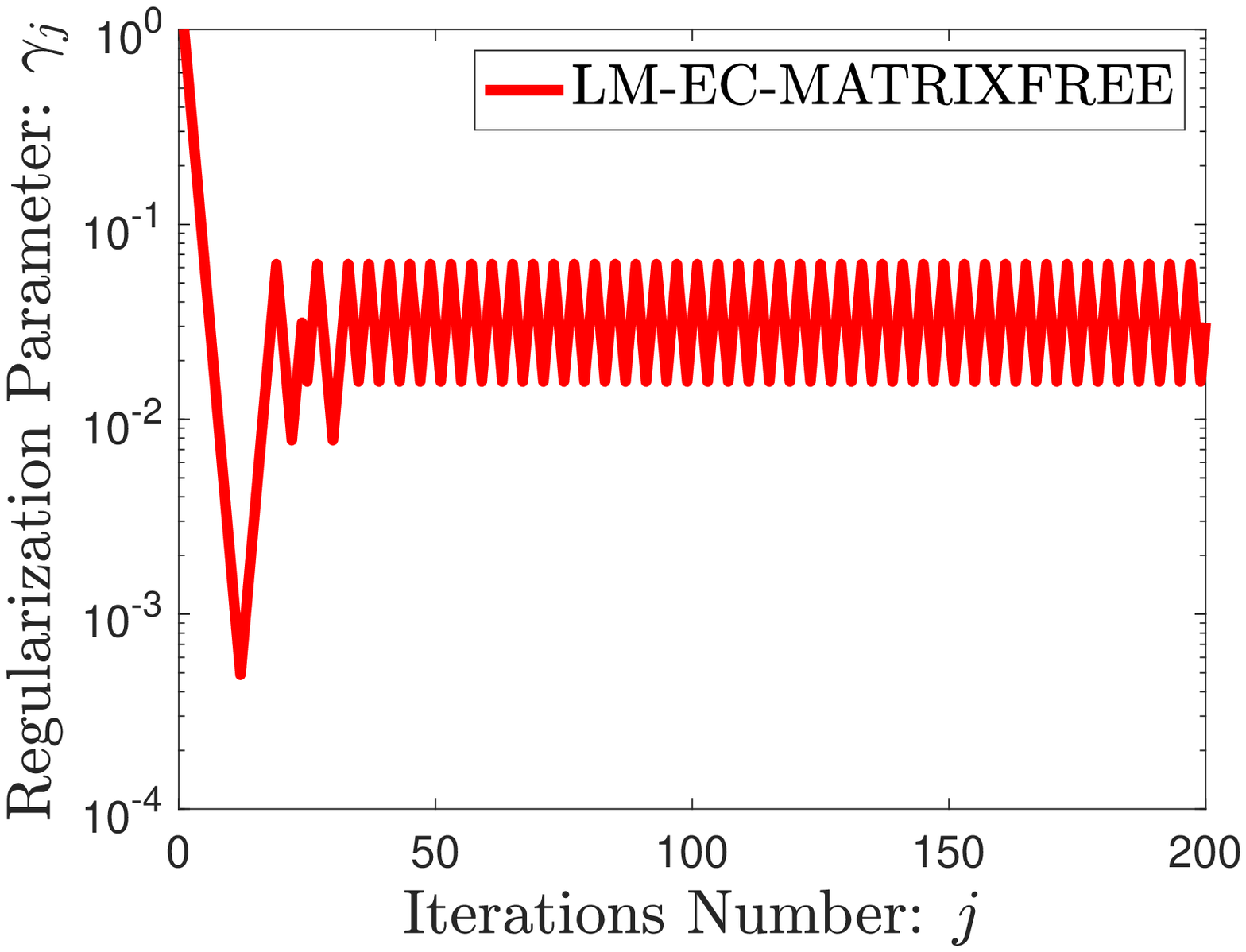}
\caption{Convergence plots for the PDE-constrained optimization problem 
using matrix-free solvers.} \label{fig:pde:op}
\end{figure}

\subsection{A coupled ODE-PDE nonlinear inverse problem}
\label{subsec:numerics:gwater}

We finally compare our methods on the \texttt{G\_Water} problem described by 
Schittkowski~\cite[Section 6.3]{schittkowski1997parameter}. This inverse problem 
features a nonlinear least-squares objective and nonlinear constraints formed 
by a discretized coupled ODE-PDE system modeling acidification of groundwater 
pollution. The resulting infinite-dimensional optimization problem is:
\[
\left\{
\begin{array}{lrll}
\min\limits_{c_m,c_{im}} & f([c_m, c_{im}]):=& \frac 12\left\|c_m(40,t)-\frac{D_m}{V_m} \frac{\partial c_m}{\partial x}(40,t)-\hat h\right\|^2 &  \\
\text{subject to} & \theta_m \frac{\partial c_m}{\partial t}(x,t) & =  \theta_m D_m \frac{\partial^2 c_m}{\partial x^2}(x,t)-\theta_m V_m \frac{\partial c_m}{\partial x}(x,t) &
  \\
& \theta_{im} \frac{\partial c_{im}}{\partial t}(x,t) & =  \alpha(c_m(x,t)-c_{im}(x,t)) & \\
& c_m(0,t)-\frac{D_m}{V_m}\frac{\partial c_m}{\partial x}(0,t) & =  \left\{\begin{array}{ll} 5800, & \text{if }t<0.01042 \\ 0, & \text{otherwise}\end{array}\right. & \\
& c_m(80,t)+\frac{D_m}{V_m} \frac{\partial c_m}{\partial x}(80,t) & =  0, &
\end{array}
\right.
\]
where $\{\theta_m,\theta_{im},V_m,\alpha,D_m\}$ are parameters, $c_m$ and 
$c_{im}$ are the functions to determine, $\hat h$ is the infinite-dimensional 
observation vector, and $\Omega\times T = (0,80)\times (0,2.55]$ is the domain. 
The initial conditions are that $c_{im}(x,0) = 0$ and $c_m(x,0)=0$.



In order to generate the measurements, we took the values of the parameters 
corresponding to the lowest residual in the reported results 
in~\cite{schittkowski1997parameter}, and simulated the PDE using a spatial 
discretization with $n_x = 16$ using the \texttt{ode15s} MATLAB function, which 
generated a time discretization of $n_t=197$. Since there is an observation at 
each time point, we obtained a residual vector (corresponding to $F(x)$ 
in~\eqref{eq:prb}) of length $197$. We used the measurements of 
$\tilde h(c_m):= c_m(40,t)-\frac{D_m}{V_m} \frac{\partial c_m}{\partial x}(40,t)$ 
at each of the time points to compute the vector of $\hat{h}$. Ultimately, the 
dimension of the constraints vector $C(x)$ is $m=6304$ and the number of the 
unknown variables was $d=6309$.

We emphasize that this problem is highly nonlinear and difficult to solve, 
especially with a low-quality starting point. Since we did not find 
recommendations for starting points in the literature, we ran our two 
matrix-free solvers from 150 starting points chosen uniformly at random: the
final residual for $\|C(x)\|+\|\tilde Wg\|$ was less than $10$ for 23 of the runs, 
while the final feasibility measure $\|C(x)\|$ was less than $10^{-1}$ on 67 of 
the runs. The initial values of $\|C(x)\|+\|\tilde Wg\|$ were typically of the 
order of $[10^4,10^6]$, which illustrates the challenges posed by this 
problem.
Figure~\ref{fig:pdeode} plots the progress along the iterations of both 
algorithms with the lowest final value of $\|C(x)\|+\|\tilde Wg\|$. Note that 
the norm of the constraint quickly drops to tolerance, which indicates that 
our methods produce iterates that end to respect the physics of the problem.

\begin{figure}[!h]
\centering
\includegraphics[scale=0.34]{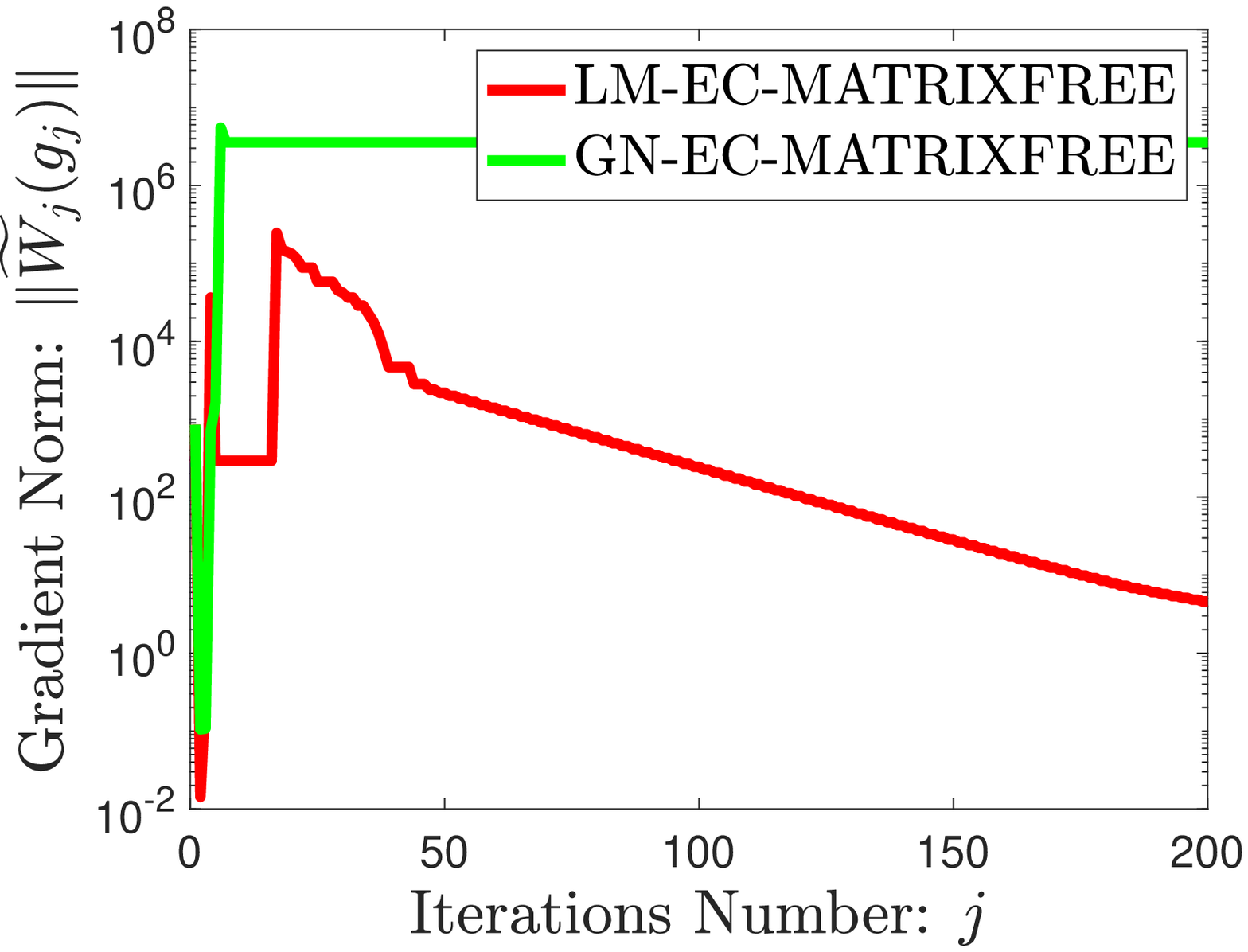}
\includegraphics[scale=0.34]{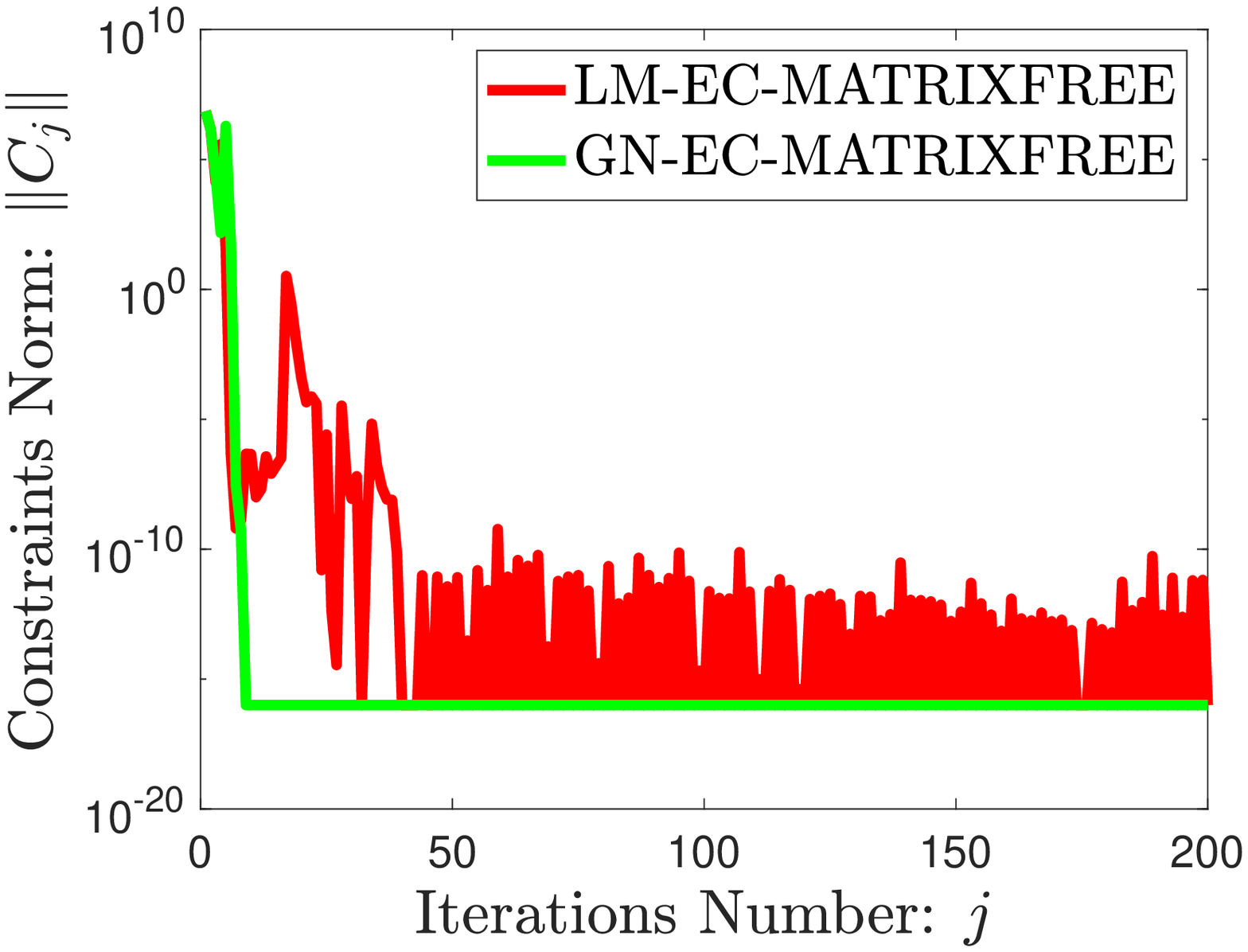}
\includegraphics[scale=0.34]{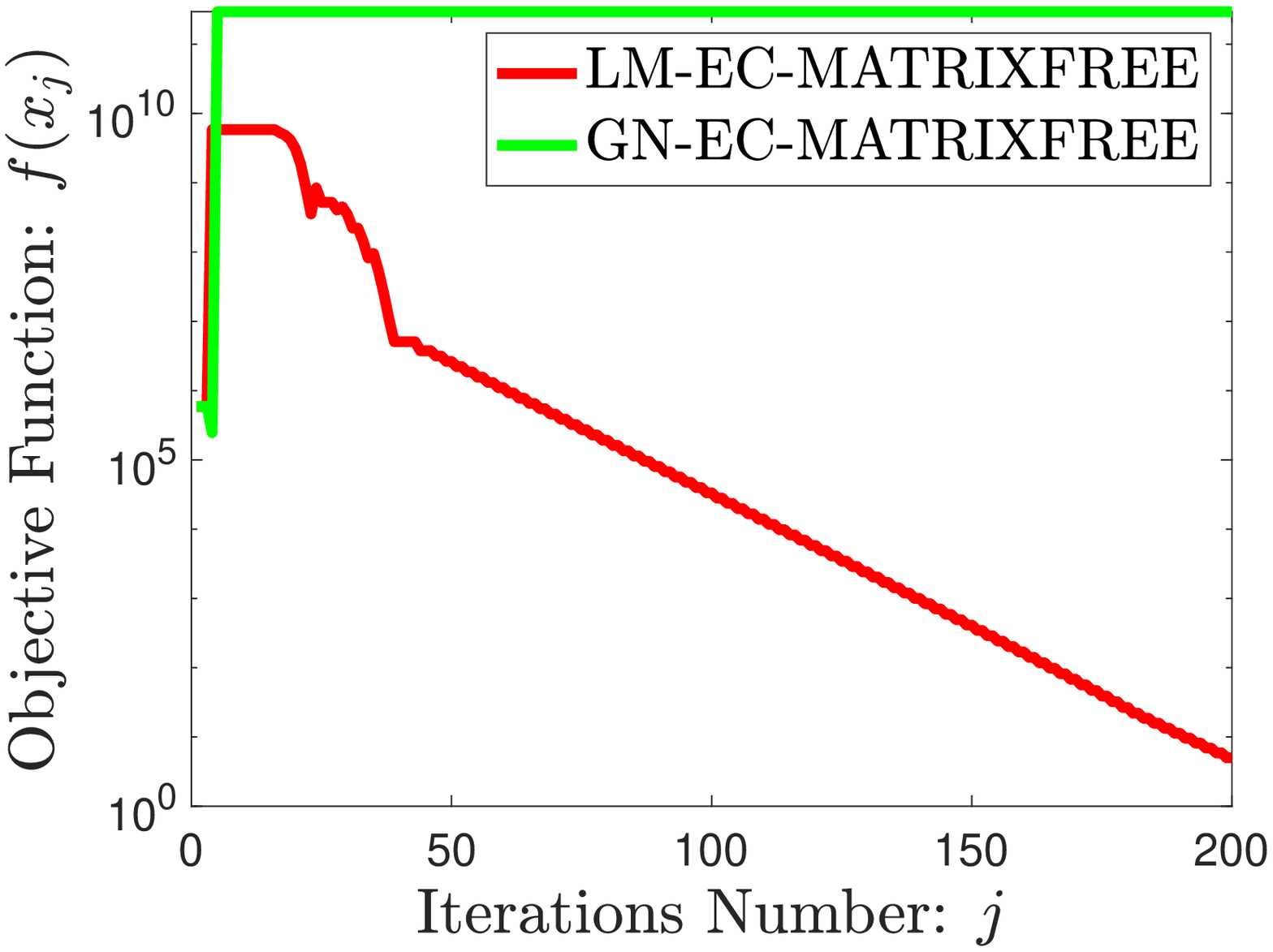}
\includegraphics[scale=0.34]{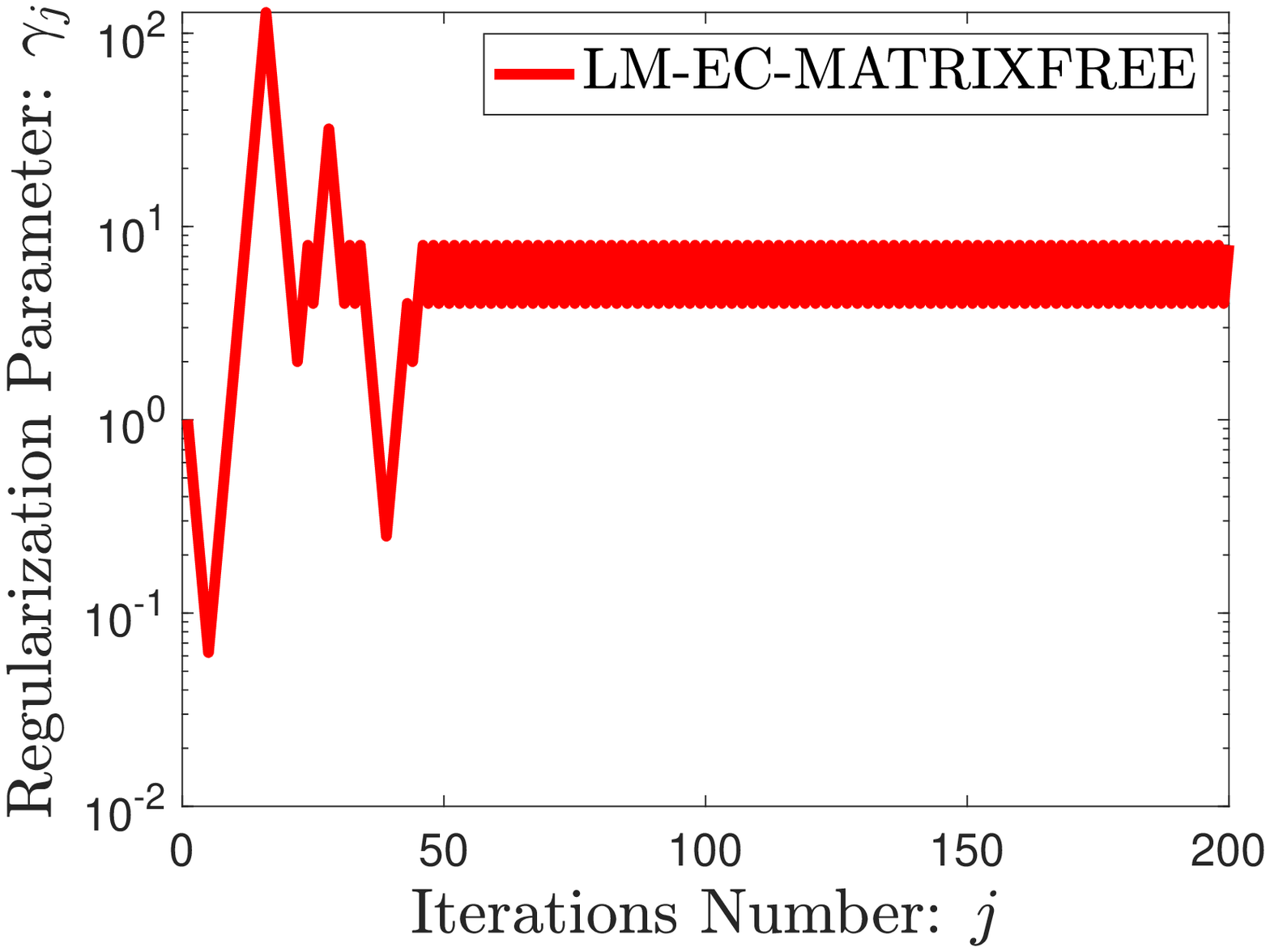}
\caption{Performance for the first $100$ iterations on the inverse coupled PDE-ODE problem.} \label{fig:pdeode}
\end{figure}

\section{Conclusion}
\label{sec:conc}

We have proposed and analyzed a nonmonotone composite-step Levenberg-Marquardt 
algorithm for the solution of equality-constrained least-squares problems. Our 
approach allows for approximate solutions of the subproblems computed by 
iterative linear algebra, and is endowed with global convergence guarantees 
through a nonmonotone step acceptance rule. Our numerical experiments showed 
that our method is competitive with trust-region approaches, and converges on 
a variety of experiments from data assimilation to PDE-constrained optimization. 

Our theoretical analysis focuses on global convergence, yet complexity results 
have become increasingly popular in the optimization community. Deriving 
worst-case bounds on the number of Hessian-vector products required to reach an 
approximate solution of the problem is a potential follow-up of this work.
Besides, our applications of interest such as data assimilation and 
PDE-constrained optimization, the measurements (and sometimes the models 
themselves) can be noisy, which significantly hardens the optimization task. 
Incorporating uncertainty into our framework thus represents an interesting 
avenue for future research.

\section*{Acknowledgements}
The authors would like to thank the guest editors as well as two anonymous referees for their 
insightful comments.
\bibliographystyle{plain}
\addcontentsline{toc}{chapter}{Bibliography}
\bibliography{refs}

\begin{thebibliography}{10}

\bibitem{HAntil_DPKouri_MDLacasse_DRidzal_2016EDS}
H.~Antil, D.~P. {Kouri}, M.-D. {Lacasse}, and D.~Ridzal, editors.
\newblock {\em Frontiers in PDE-Constrained Optimization}, volume 163 of {\em
  The IMA Volumes in Mathematics and its Applications}.
\newblock Springer, New York, NY, USA, 2016.

\bibitem{asch2016}
M.~Asch, M.~Bocquet, and M.~Nodet.
\newblock {\em Data Assimilation: {M}ethods, Algorithms, and Applications}.
\newblock SIAM, 2016.

\bibitem{RBehling_AFischer_2012}
R.~Behling and A.~Fischer.
\newblock A unified local convergence analysis of inexact constrained
  {L}evenberg-{M}arquardt methods.
\newblock {\em Optim. Lett.}, 6:927--940, 2012.

\bibitem{EBergou_YDiouane_VKungurtsev_2017}
E.~Bergou, Y.~Diouane, and V.~Kungurtsev.
\newblock Convergence and complexity analysis of a {L}evenberg-{M}arquardt
  algorithm for inverse problems.
\newblock {\em J. Optim. Theory Appl.}, 185:927--944, 2020.

\bibitem{EBergou_SGratton_LNVicente_2016}
E.~Bergou, S.~Gratton, and L.~N. {Vicente}.
\newblock Levenberg-{M}arquardt methods based on probabilistic gradient models
  and inexact subproblem solution, with application to data assimilation.
\newblock {\em {SIAM/ASA} J. Uncertain. Quantif.}, 4:924--951, 2016.

\bibitem{SBoyd_LVandenberghe_2018}
S.~Boyd and L.~Vandenberghe.
\newblock {\em Introduction to Applied Linear Algebra - Vectors, Matrices and
  Least Squares}.
\newblock Cambridge University Press, Cambridge, United Kingdom, 2018.

\bibitem{SCTChoi_CCPaige_MASaunders_2011}
S.-C.~T. {Choi}, C.~C. {Paige}, and M.~A. {Saunders}.
\newblock {MINRES-QLP}: {A} {K}rylov subspace method for indefinite or singular
  symmetric systems.
\newblock {\em SIAM J. Sci. Comput.}, 33:1810--1836, 2011.

\bibitem{JrJEDennis_MElAlem_MCMaciel_1997}
J.~E. Dennis, M.~El-Alem, and M.~C. Maciel.
\newblock A global convergence theory for general trust-region-based algorithms
  for equality constrained optimization.
\newblock {\em SIAM J. Optim.}, 7:177--207, 1997.

\bibitem{Dolan_2002}
Elizabeth~D Dolan and Jorge~J Mor{\'e}.
\newblock Benchmarking optimization software with performance profiles.
\newblock {\em Mathematical Programming}, 91(2):201--213, 2002.

\bibitem{Haber07modelproblems}
E.~Haber and L.~Hanson.
\newblock Model problems in pde-constrained optimization.
\newblock Technical report, 2007.

\bibitem{PCHansen_VPereyra_GScherer_2012}
P.~C. {Hansen}, V.~Pereyra, and G.~Scherer.
\newblock {\em Least Squares Data Fitting with Applications}.
\newblock Johns Hopkins University Press, Baltimore, MD, USA, 2012.

\bibitem{heinkenschloss2014matrixfreeSQP}
M.~Heinkenschloss and D.~Ridzal.
\newblock A matrix-free trust-region sqp method for equality constrained
  optimization.
\newblock {\em SIAM J. Optim.}, 24(3):1507--1541, 2014.

\bibitem{heinkenschloss2002analysis}
M.~Heinkenschloss and L.~N. Vicente.
\newblock Analysis of inexact trust-region sqp algorithms.
\newblock {\em SIAM J. Optim.}, 12(2):283--302, 2002.

\bibitem{WHock_KSchittkowski_1980}
W.~Hock and K.~Schittkowski.
\newblock Test examples for nonlinear programming codes.
\newblock {\em J. Optim. Theory Appl.}, 30:127--129, 1980.

\bibitem{izmailov2019globally}
A.~F. Izmailov, M.~V. Solodov, and E.~Uskov.
\newblock A globally convergent {L}evenberg--{M}arquardt method for
  equality-constrained optimization.
\newblock {\em Comput. Optim. Appl.}, 72(1):215--239, 2019.

\bibitem{KLevenberg_1944}
K.~Levenberg.
\newblock A method for the solution of certain problems in least squares.
\newblock {\em Quart. Appl. Math.}, 2:164--168, 1944.

\bibitem{ZFLi_MROsborne_TPrvan_2002}
Z.~F. Li, M.~R. Osborne, and T.~Prvan.
\newblock Adaptive algorithm for constrained least-squares problems.
\newblock {\em J. Optim. Theory Appl.}, 114:423--441, 2002.

\bibitem{lorenz63}
E.~N. Lorenz.
\newblock Deterministic non periodic flow.
\newblock {\em J. Atmos. Sci}, 20(2):130--141, 1963.

\bibitem{DMarquardt_1963}
D.~Marquardt.
\newblock An algorithm for least-squares estimation of nonlinear parameters.
\newblock {\em SIAM J. Appl. Math.}, 11:431--441, 1963.

\bibitem{NMarumo_TOkuno_ATakeda_2020}
N.~Marumo, T.~Okuno, and A.~Takeda.
\newblock Constrained {L}evenberg-{M}arquardt method with global complexity
  bound.
\newblock arXiv:2004.08259, 2020.

\bibitem{JNocedal_SJWright_2006}
J.~Nocedal and S.~J. {Wright}.
\newblock {\em Numerical {O}ptimization}.
\newblock Springer Series in Operations Research and Financial Engineering.
  Springer-Verlag, New York, second edition, 2006.

\bibitem{DOrban_ASSiquiera_2020}
D.~Orban and A.~S. {Siqueira}.
\newblock A regularization method for constrained nonlinear least squares.
\newblock {\em Comput. Optim. Appl.}, 76:961--989, 2020.

\bibitem{Schittkowski_1987}
K.~Schittkowski, editor.
\newblock {\em More Test Examples for Nonlinear Programming Codes}.
\newblock Springer-Verlag, Berlin, Heidelberg, 1987.

\bibitem{schittkowski1997parameter}
K.~Schittkowski.
\newblock Parameter estimation in one-dimensional time-dependent partial
  differential equations.
\newblock {\em Optim. Methods Softw.}, 7(3-4):165--210, 1997.

\bibitem{TSteihaug_1983}
T.~Steihaug.
\newblock {The conjugate gradient method and trust regions in large scale
  optimization.}
\newblock {\em SIAM J. Numer. Anal.}, 20:626--637, 1983.

\bibitem{ATarantola_2005}
A.~Tarantola.
\newblock {\em Inverse Problem Theory and Methods for Model Parameter
  Estimation}.
\newblock SIAM, Philadelphia, 2005.

\bibitem{YTremolet_2007}
Y.~Tr\'{e}molet.
\newblock Model-error estimation in {4D-Var}.
\newblock {\em Quarterly Journal of the Royal Meteorological Society},
  133:1267--1280, 2007.

\bibitem{MUlbrich_SUlbrich_2003}
M.~Ulbrich and S.~Ulbrich.
\newblock Nonmonotone trust region methods for nonlinear equality constrained
  optimization without a penalty function.
\newblock {\em Math. Program.}, 95:103--135, 2003.

\end{thebibliography}

\end{document}